\documentclass[12pt]{amsart}
\usepackage{amsmath}
\usepackage{amssymb}
\usepackage{amsfonts}
\usepackage{amsthm}
\usepackage{enumerate}
\usepackage[mathscr]{eucal}
\usepackage{verbatim}
\usepackage{amsthm}
\usepackage{amscd}
\usepackage[mathscr]{eucal}
\usepackage{appendix}
\usepackage{tikz}

\numberwithin{equation}{section}

\newtheorem{innercustomgeneric}{\customgenericname}
\providecommand{\customgenericname}{}

\newcommand{\newcustomtheorem}[2]{\newenvironment{#1}[1]
  {\renewcommand\customgenericname{#2}
   \renewcommand\theinnercustomgeneric{##1}\innercustomgeneric}{\endinnercustomgeneric}}

\newcustomtheorem{customthm}{Theorem}

\newcommand{\newcustomlemma}[2]{\newenvironment{#1}[1]
  {\renewcommand\customgenericname{#2}
   \renewcommand\theinnercustomgeneric{##1} \innercustomgeneric}{\endinnercustomgeneric}}

\newcustomlemma{customlemma}{Lemma}

\newcustomlemma{customproposition}{Proposition}

\theoremstyle{plain}
\newtheorem{theorem}{Theorem}[section]
\newtheorem{lemma}[theorem]{Lemma}

\newtheorem{proposition}[theorem]{Proposition}

\theoremstyle{remark}

\theoremstyle{definition}

\newcommand{\bnu}{\begin{enumerate}}
\newcommand{\enu}{\end{enumerate}}
\newcommand{\veq}{\mathrel{\rotatebox{90}{$=$}}}

\newcommand{\q}{\quad}
\newcommand{\qq}{\qquad}

\newcommand{\eset}{\emptyset}

\newcommand{\al}{\alpha}

\newcommand{\ga}{\gamma}

\newcommand{\om}{\omega}
\newcommand{\Om}{\Omega}

\newcommand{\la}{\lambda}

\newcommand{\ep}{\epsilon}

\newcommand{\de}{\delta}

\newcommand{\bbz}{\mathbb{Z}}

\newcommand{\bbzn}{\mathbb Z^n}

\newcommand{\xx}{\mathbf x}

\newcommand{\bbr}{\mathbb{R}}

\newcommand{\bbrn}{\mathbb R^n}
\newcommand{\rn}{\mathbb{R}^n}

\newcommand{\bbc}{\mathbb{C}}

\newcommand{\bbn}{\mathbb{N}}

\newcommand{\NN}{\mathcal{N}}
\newcommand{\UU}{\mathcal{U}}

\newcommand{\II}{\mathcal{I}}
\newcommand{\JJ}{\mathcal{J}}

\newcommand{\LL}{\mathcal{L}}

\newcommand{\xxxi}{\vec{{\xi}\;}}

\newcommand{\yyy}{\vec{{y}}}
\newcommand{\zzz}{\vec{{z}}}

\newcommand{\0}{\vec{\boldsymbol{0}}}
\newcommand{\nnu}{\vec{{\nu}}}

\newcommand{\supp}{\mathrm{supp}}

\newcommand{\p}{\partial}
\newcommand{\nf}{\infty}

\newcommand{\tf}{\tfrac}
\newcommand{\wh}{\widehat}
\newcommand{\wt}{\widetilde}

\newcounter{question}

\newcommand{\bpf}{\begin{proof}}

\newcommand{\epf}{\end{proof}}

\allowdisplaybreaks

\makeindex         

\begin{document}

\author{Danqing He}
\address{D. He, School of Mathematical Sciences,
Fudan University, Shanghai,  200433, P.R. China}
\email{hedanqing@fudan.edu.cn}

\author{Bae Jun Park}
\address{B. Park, Department of Mathematics, Sungkyunkwan University, Suwon 16419, Republic of Korea}
\email{bpark43@skku.edu}

\thanks{ D. He is supported by  National Key R$\&$D Program of China (No. 2021YFA1002500), NNSF of China (No. 12161141014),  and Natural Science Foundation of Shanghai (No. 22ZR1404900).
B. Park is supported by NRF grant 2022R1F1A1063637.}

\title{Improved estimates for bilinear rough singular integrals}

\subjclass[2010]{Primary 42B20, 42B99}

\begin{abstract} 
We study bilinear rough singular integral operators $\LL_{\Omega}$ associated with a function $\Omega$ on the sphere $\mathbb{S}^{2n-1}$.
 In the recent work of Grafakos, He, and Slav\'ikov\'a \cite{Gr_He_Sl2020}, they showed that 
$\LL_{\Omega}$ is bounded from $L^2\times L^2$ to $L^1$, provided that $\Omega\in L^q(\mathbb{S}^{2n-1})$ for $4/3<q\le \infty$ with mean value zero. 
In this paper, we provide a generalization of  their result.
We actually prove $L^{p_1}\times L^{p_2}\to L^p$ estimates for $\LL_{\Omega}$ under the assumption
 $$\Omega\in L^q(\mathbb{S}^{2n-1}) \q \text{ for }~\max{\Big(\;\frac{4}{3}\;,\; \frac{p}{2p-1} \;\Big)<q\le \infty}$$
 where $1<p_1,p_2\le\infty$ and $1/2<p<\infty$  with $1/p=1/p_1+1/p_2$ .
Our result improves that of Grafakos, He, and Honz\'ik \cite{Gr_He_Ho2018}, in which the more restrictive condition $\Omega\in L^{\infty}(\mathbb{S}^{2n-1})$ is required for the $L^{p_1}\times L^{p_2}\to L^p$ boundedness.
\end{abstract}

\maketitle


\section{Introduction}
 The study of rough singular integral operators dates back to the work of Calder\'on and Zygmund \cite{Ca_Zy1956}. They proved that the operator $\LL_{\Omega}$, defined by
 $$\LL_{\Omega}f(x):=p.v.\int_{\bbrn}\frac{\Omega(y/|y|)}{|y|^n}f(x-y)dy,$$
 is bounded on $L^p(\bbrn)$ for $1<p<\infty$ where
 $\Omega\in L\log L(\mathbb{S}^{n-1})$ with vanishing integral, namely $\int_{\mathbb{S}^{n-1}}\Omega d\sigma =0$.
 This result was refined by Coifman and Weiss \cite{Co_We1977} and Connett \cite{Co1979}, using the weaker condition that $\Omega$ belongs to the Hardy space $ H^1(\mathbb{S}^{n-1})$.
 The weak type $(1,1)$ boundedness for $\LL_{\Omega}$ in small dimensions was established by Christ and Rubio de Francia \cite{Ch_Ru1988} and independently by Hofmann \cite{Ho1988}, both inspired by the work of Christ \cite{Ch1988}. This was later extended to arbitrary dimensions by Seeger \cite{Se1996}.
 
 \hfill
 
 Coifman and Meyer \cite{Co_Me1975} first studied bilinear singular integrals. Suppose $\Omega$ is an integrable function on $\mathbb{S}^{2n-1}$ with $\int_{\mathbb{S}^{2n-1}}\Omega d\sigma=0$.
 We define the corresponding bilinear rough singular integral operator $\LL_{\Omega}$ (which is denoted as in the linear setting without risk of confusion as the linear counterpart will not appear in the sequel) by
 \begin{equation*}
\LL_{\Om}\big(f_1,f_2\big)(x):=p.v. \int_{(\bbrn)^2}{K(y_1,y_2)f_1(x-y_1)f_2(x-y_2)}~dy_1 dy_2
\end{equation*} 
 where 
 \begin{equation}\label{kerneldef}
K(y_1,y_2):=\frac{\Omega((y_1,y_2)')}{|(y_1,y_2)|^{2n}}, \qquad  (y_1,y_2) \neq (0,0)
\end{equation} for $ (y_1,y_2)':=\frac{(y_1,y_2)}{|(y_1,y_2)|}\in \mathbb{S}^{2n-1}$.
Then Grafakos, He, and Honz\'ik \cite{Gr_He_Ho2018} established the $L^{p_1}\times L^{p_2}\to L^p$ boundedness for $\LL_{\Omega}$.
\begin{customthm}{A}\cite{Gr_He_Ho2018}\label{thma}
Let $1<p_1,p_2<\infty$ and $1/p=1/p_1+1/p_2$.
Suppose that $\Omega\in L^{\infty}(\mathbb{S}^{2n-1})$ and $\int_{\mathbb{S}^{2n-1}}{\Omega}d\sigma =0$. Then there exists a constant $C>0$ such that
\begin{equation*}
\Vert \LL_{\Omega}\Vert_{L^{p_1}\times L^{p_2}\to L^p}\le C \Vert \Omega\Vert_{L^{\infty}(\mathbb{S}^{2n-1})}.
\end{equation*}
\end{customthm}
 It was first proved that
\begin{equation}\label{l2l2l1}
\Vert \LL_{\Omega}\Vert_{L^{2}\times L^{2}\to L^1}\lesssim \Vert \Omega\Vert_{L^{2}(\mathbb{S}^{2n-1})},
\end{equation} using a wavelet decomposition of Daubechies, and then apply the bilinear Calder\'on-Zygmund theory in \cite{Gr_To2002} to extend it to the indices $1<p_1,p_2<\infty$.
In the recent paper of Grafakos, He, and Slav\'ikov\'a \cite{Gr_He_Sl2020}, the estimate (\ref{l2l2l1}) has been improved by replacing $\Vert \Omega\Vert_{L^2(\mathbb{S}^{2n-1})}$ by $\Vert \Omega\Vert_{L^q(\mathbb{S}^{2n-1})}$ for $q>4/3$, as an application of the following theorem.
 \begin{customthm}{B}\cite{Gr_He_Sl2020, Sl2021}\label{thmb}
 Let $1< r<4$. Set $M$ to be a positive integer satisfying 
$$
M>\frac{2n}{4-r}.
$$
Suppose that $\mathrm{m}\in L^r((\bbrn)^{2})\bigcap \mathscr{C}^{M}((\bbrn)^{2})$ with 
$$\big\Vert \partial^{\alpha}\mathrm{m} \big\Vert_{L^{\infty}((\bbrn)^{2})}\le D_0<\infty, \q \text{for }~ |\alpha|\le M.$$
Then the bilinear operator $T_{\mathrm{m}}$ defined by
$$T_{\mathrm{m}}(f_1,f_2)(x):=\int_{(\bbrn)^2} \mathrm{m}(\xi_1,\xi_2)\wh{f_1}(\xi_1)\wh{f_2}(\xi_2)e^{2\pi i\langle x,\xi_1+\xi_2\rangle} d\xi_1 d\xi_2,$$
where $\wh{f}(\xi):=\int_{\bbrn}f(x)e^{-2\pi i\langle x,\xi\rangle}dx$ denotes the Fourier transform of $f$,
is bounded from $L^2\times L^2$ to $L^1$. Moreover,  we have
\begin{equation*}
\big\Vert T_{\mathrm{m}}\big\Vert_{L^2\times L^2\to L^{1}}\lesssim D_0^{1-\frac{r}{4}} \Vert \mathrm{m}\Vert_{L^r((\bbrn)^{2})}^{\frac{r}{4}}.
\end{equation*} 
Conversely, for $r\ge 4$, there is a function $\mathrm{m}\in L^r((\bbrn)^2)\cap \mathscr{C}^{\infty}((\bbrn)^2) $ such that $T_{\mathrm{m}}$ does not map $L^2\times L^2\to L^1$.
 \end{customthm}
 We remark that the $L^{2}\times L^2\to L^1$ estimate in Theorem \ref{thmb} can be generalized to $L^{p_1}\times L^{p_2}\to L^{p}$ for all indices $p_1,p_2,p$ satisfying $2\le p_1,p_2\le \infty$, $1\le p\le 2$, and $1/p=1/p_1+1/p_2$, using duality and interpolation. 
 As a consequence, we actually have the following result.
  \begin{customthm}{C}\cite{Gr_He_Sl2020}\label{thmc}
 Let $4/3<q\le \infty$ and assume that $\Omega\in L^q(\mathbb{S}^{2n-1})$ with $\int_{\mathbb{S}^{2n-1}}\Omega d\sigma=0$.
 Then there exists a constant $C>0$ such that
 \begin{equation}\label{mainestimate}
 \Vert \LL_{\Omega}\Vert_{L^{p_1}\times L^{p_2}\to L^p}\le C \Vert \Omega\Vert_{L^q(\mathbb{S}^{2n-1})}
 \end{equation}
 whenever $2\le p_1,p_2\le \infty$, $1\le p\le 2$, and $1/p=1/p_1+1/p_2$.
 \end{customthm}
  The condition $q>4/3$ is corresponding to $r<4$ in Theorem \ref{thmb} with the relationship $1/q+1/r=1$.
 
It is natural to ask for the optimal range of $q$ for which the boundedness (\ref{mainestimate}) holds.
In \cite{Gr_He_Sl2019}, Grafakos, He, and Slav\'kov\'a show that there exists $\Omega\in L^q(\mathbb{S}^{2n-1})$ with mean value zero such that $\LL_{\Omega}$ is not bounded from $L^{p_1}\times L^{p_2}$ to $L^p$ for $1\le p_1,p_2<\infty$ and $1/2\le p<1$ with $1/p=1/p_1+1/p_2$ if $q$ is near $1$ ( explicitly, $1\le q<\frac{2n-1}{2n-1/p}$ ). This is different from the linear case in which the $L^p$ boundedness holds if $\Omega\in L^q(\mathbb{S}^{n-1})$ for $1<q\le \infty$ as $L^q(\mathbb{S}^{n-1})\subset L\log L(\mathbb{S}^{n-1})\subset H^1(\mathbb{S}^{n-1})$.
However, the estimate (\ref{mainestimate}) remains still open for $\frac{2n-1}{2n-1/p}\le q\le \frac{4}{3}$.

\hfill

In this paper, we provide a generalization of Theorem \ref{thmc} in the whole range $1<p_1,p_2\le\infty$ and $1/2<p<\infty$ with $1/p=1/p_1+1/p_2$, which improves the result in Theorem \ref{thma}.
The main result is as follows:
\begin{theorem}\label{mainthm}
  Let $1<p_1,p_2\le\infty$ and $1/2< p<\infty$ with $1/p=1/p_1+1/p_2$.
Suppose that $$\max{\Big( \frac{4}{3},\frac{p}{2p-1} \Big)}<q\le \infty,$$
and $\Omega \in L^q (\mathbb S^{2n-1})$ with $\int_{\mathbb S^{2n-1}} \Omega \, d\sigma=0$. Then the estimate (\ref{mainestimate}) holds
\end{theorem}

For comparison with Theorems \ref{thma} and \ref{thmc}, we refer to Figure \ref{fig1}.
It seems that
the bilinear Caler\'on-Zygmund theory in establishing Theorem \ref{thma} 
is not applicable to the case when $\Omega\in L^q$ for $q\not= \infty$, which requires a more delicate analysis.
Actually we develop a bilinear Calder\'on-Zygmund argument adapted to bilinear rough singular integrals, which works effectively combined with  the dyadic decomposition in \cite{Du_Ru} and has  potential applications to other related operators.
This strategy however does not yield the boundedness \eqref{mainestimate} of the endpoints, say, when $p_1=\infty$. To overcome this obstacle, we need a  decay in the local $L^2$ cases improving the duality result in Theorem~\ref{thmc}, which follows from a refinement of the wavelet argument developed in \cite{Gr_He_Ho2018} and \cite{Gr_He_Sl2020}; see Proposition~\ref{05032} for the accurate formulation.
In summary, we establish a decay at $L^2\times L^\nf\to L^2$, and obtain arbitrarily slow growth at $L^{1}\times L^{p_2}\to L^{p,\infty}$ for $1/p=1+1/p_2$, and $L^{\infty}\times L^{\infty}\to BMO$.

 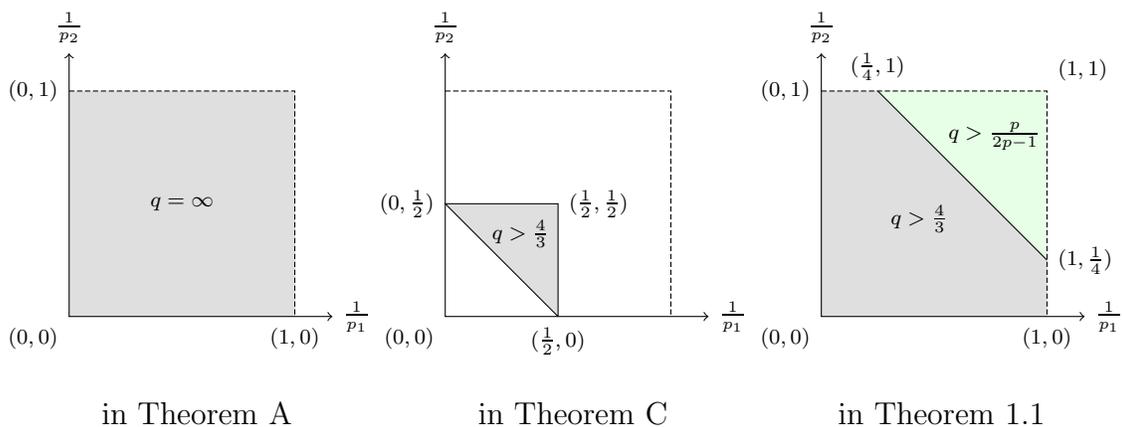
\begin{figure}[h]\label{fig1}
\begin{tikzpicture}
\path[fill=gray!25] (0,0)--(0,3)--(3,3)--(3,0)--(0,0);
\draw [<->] (0,3.5)--(0,0)--(3.5,0);
\draw[dash pattern= { on 2pt off 1pt}] (0,3)--(3,3)--(3,0);
\node [below left] at (0,0) {\tiny$(0,0)$};
\node [below] at (3,0) {\tiny$(1,0)$};
\node [left] at (0,3) {\tiny$(0,1)$};
\node  at (1.5,1.5) {\tiny$q=\infty$};
\node [right] at (3.5,0) {\tiny${\frac{1}{p_1}}$};
\node [above] at (0,3.5) {\tiny${\frac{1}{p_2}}$};
\node [below] at (1.7,-1) {in Theorem \ref{thma}};
\path[fill=gray!25] (5,1.5)--(6.5,1.5)--(6.5,0)--(5,1.5);
\draw [<->] (5,3.5)--(5,0)--(8.5,0);
\draw[dash pattern= { on 2pt off 1pt}] (5,3)--(8,3)--(8,0);
\draw[-] (6.5,0)--(6.5,1.5)--(5,1.5)--(6.5,0);
\node  at (6,1.1) {\tiny$q>\frac{4}{3}$};
\node [left] at (5,1.5) {\tiny$(0,\frac{1}{2})$};
\node [below left] at (5,0) {\tiny$(0,0)$};
\node [below] at (6.5,0) {\tiny$(\frac{1}{2},0)$};
\node [right] at (6.5,1.5) {\tiny$(\frac{1}{2},\frac{1}{2})$};
\node [right] at (8.5,0) {\tiny${\frac{1}{p_1}}$};
\node [above] at (5,3.5) {\tiny${\frac{1}{p_2}}$};
\node [below] at (6.7,-1) {in Theorem \ref{thmc}};

\path[fill=gray!25] (10,0)--(13,0)--(13,0.75)--(10.75,3)--(10,3)--(10,0);
\path[fill=green!10] (13,0.75)--(13,3)--(10.75,3)--(13,0.75);
\draw [<->] (10,3.5)--(10,0)--(13.5,0);
\draw[dash pattern= { on 2pt off 1pt}] (10,3)--(13,3)--(13,0);
\draw[-] (10.75,3)--(13,0.75);
\node [below left] at (10,0) {\tiny$(0,0)$};
\node [below] at (13,0) {\tiny$(1,0)$};
\node [left] at (10,3) {\tiny$(0,1)$};
\node [above] at (10.75,3) {\tiny$(\frac{1}{4},1)$};
\node [right] at (13,0.75) {\tiny$(1,\frac{1}{4})$};
\node [above right] at (13,3) {\tiny$(1,1)$};
\node  at (11.3,1.3) {\tiny $q>\frac{4}{3}$};
\node  at (12.3,2.4) {\tiny $q>\frac{p}{2p-1}$};
\node [right] at (13.5,0) {\tiny${\frac{1}{p_1}}$};
\node [above] at (10,3.5) {\tiny${\frac{1}{p_2}}$};
\node [below] at (11.6,-1) {in Theorem \ref{mainthm}};
\end{tikzpicture}
\caption{The range of $q$ for the estimate (\ref{mainestimate})}
\end{figure}

\hfill

Section \ref{preliminary} contains some preliminary materials that will be crucial tools in the proof of Theorem \ref{mainthm}.
We set up the structure of the proof of Theorem \ref{mainthm} in Section \ref{reduction}. The argument in this section actually appeared in \cite{Gr_He_Ho2018} and matters will, in turn, reduce to operators with smooth kernels, which come up in a dyadic decomposition of the kernel $K$. We complete the proof in the remaining sections, which are actually the main parts of this paper, by mostly dealing with end-point estimates of weak-type and $BMO$-type, and by interpolating such boundedness results.

 \hfill
 
{\textit{Notation.}}
Let $\bbn$ and $\bbz$ be the sets of all natural numbers and all integers, respectively. We use the symbol $A\lesssim B$ to indicate that $A\leq CB$ for some constant $C>0$ independent of the variable quantities $A$ and $B$, and $A\sim B$ if $A\lesssim B$ and $B\lesssim A$ hold simultaneously.
For each cube $Q$ in $\bbrn$, let $\ell(Q)$ and $c_Q$ mean the side-length and the center of $Q$, respectively. Let $Q^*$ be the concentric dilation of $Q$ with $\ell(Q^*)=10^2 \sqrt{n}\ell(Q)$ and denote by  $\chi_Q$ the characteristic function of $Q$.
For simplicity, we adopt the notation $\yyy:=(y_1,y_2)\in (\bbrn)^2$, $\zzz:=(z_1,z_2)\in (\bbrn)^2$,  $\xxxi:=(\xi_1,\xi_2)\in (\bbrn)^2$, and $\nnu:=(\nu_1,\nu_2)\in (\bbzn)^2$.

\section{Preliminaries}\label{preliminary}

\subsection{Maximal inequality}
Let $\mathcal{M}$ denote the Hardy-Littlewood maximal operator, defined by
$$\mathcal{M}f(x):=\sup_{Q:x\in Q}\frac{1}{|Q|}\int_Q{|f(y)|}dy$$
for a locally integrable function $f$ on $\bbrn$, where the supremum is taken over all cubes $Q$ containing $x$. 
For given $0<r<\infty$, we define $\mathcal{M}_rf:=\big( \mathcal{M}\big(|f|^r\big)\big)^{1/r}$. Then it is well-known that
\begin{equation}\label{HLmaximal}
\big\Vert \mathcal{M}_rf\big\Vert_{L^p(\bbrn)}\lesssim \Vert f\Vert_{L^p(\bbrn)}
\end{equation} whenever $0<r<p\le \infty$.

\hfill

\subsection{Interpolations}
The space $BMO(\bbrn)$ is the family of locally integrable functions $f$ on $\bbrn$ such that the norm
\begin{equation}\label{bmodef}
\Vert f\Vert_{BMO}:=\sup_{Q}\inf_{a\in\bbc}\frac{1}{|Q|}\int_Q{\big| f(x) -a \big|}dx
\end{equation}
is finite, where the supremum ranges over all cubes $Q$ in $\bbrn$. This space, introduced by John and Nirenberg \cite{Jo_Ni1961} and characterized by Fefferman \cite{Fe1971}, plays a significant role in interpolation as it may substitute $L^{\infty}$, satisfying 
\begin{equation}\label{realinter}
(L^q,BMO)_{\theta,r}=L^{p,r}=(L^q,L^{\infty})_{\theta,r}
\end{equation}
and
\begin{equation}\label{complexinter}
[L^q,BMO]_{\theta}=L^p=[L^q,L^{\infty}]_{\theta}
\end{equation} 
where $0<r\le \infty$ and $1/p=(1-\theta)/q$ for $0<\theta<1$.
Here, the symbols $(\cdot ,\cdot)_{\theta,r}$ and $[\cdot, \cdot]_{\theta}$ mean the real interpolation, so called $K$-method, and the complex method of interpolation in Calder\'on \cite{Ca1964}, respectively.
We refer to \cite{Ha1977, Ja_Jo1982} for the interpolation.
In (\ref{realinter}), the space $L^{p,r}$ is the Lorentz space, which is a generalization of the Lebesgue space $L^p$ as $L^{p,p}=L^p$, and in this paper, we are only concerned with  $L^{p,\infty}$, which is specially called weak $L^p$ space.
For $0<p\le \infty$, the space $L^{p,\infty}(\bbrn)$ is defined by the set of all measurable functions $f$ with the (quasi-)norm
$$\Vert f\Vert_{L^{p,\infty}(\bbrn)}:=\begin{cases}
\sup_{t>0}t \big|\big\{x\in\bbrn: |f(x)|>t \big\} \big|^{1/p}, & 0<p<\infty\\
\qq\qq\q \Vert f\Vert_{L^{\infty}(\bbrn)}\qq\qq\q,& p=\infty
\end{cases}.
$$
Similar to the space $BMO(\bbrn)$, interpolation results involving $L^{p,\infty}(\bbrn)$ as an end-point are useful tools to deduce the boundedness of many operators on the Lebesgue spaces. See the references \cite{Be_Sh1988, Be_Lo1976} for details. 
Indeed, in the proof of Theorem \ref{mainthm}, we will investigate $L^{p_1}\times L^{p_2}\to L^{p,\infty}$ boundedness for a certain bilinear operator to prove the strong-type estimate via interpolation. For this process, we present a bilinear version of the Marcinkiewicz interpolation theorem, which is a straightforward corollary of \cite[Theorem 3]{Ja1988} or \cite[Theorem 1.1]{Gr_Li_Lu_Zh2012}.
\begin{lemma}\cite{ Gr_Li_Lu_Zh2012,Ja1988}\label{interpollemma}
Let $0<p_{1}^0,p_{2}^0,p_{1}^1,p_{2}^1,p_1^2,p_2^2\le \infty$ and $0<p^0,p^1,p^2\le \infty$ with $1/p^j=1/p^j_{1}+1/p^j_{2}$ for $j=0,1,2$.
Suppose that $T$ is a bilinear operator having the mapping properties
\begin{equation*}
\big\Vert T(f_1,f_2) \big\Vert_{ L^{p^{j},\infty}(\bbrn)}\le M_j\Vert f_1\Vert_{L^{p_{1}^j}(\bbrn)}\Vert f_2\Vert_{L^{p_{2}^j}(\bbrn)}, \quad j=0,1,2
\end{equation*} for Schwartz functions $f_1,f_2$ on $\bbrn$.
Then for any $0<\theta_j<1$ with $\theta_0+\theta_1+\theta_2=1$, and $0<p_1,p_2,p \le \infty$  satisfying
\begin{equation*}
 \frac{1}{p_1}= \frac{\theta_0}{p_1^0}+ \frac{\theta_1}{p_1^1}+ \frac{\theta_2}{p_1^2}, \qquad \frac{1}{p_2}= \frac{\theta_0}{p_2^0}+ \frac{\theta_1}{p_2^1}+ \frac{\theta_2}{p_2^2},
\end{equation*}
\begin{equation*}
\frac{1}{p}= \frac{\theta_0}{p^0}+ \frac{\theta_1}{p^1}+ \frac{\theta_2}{p^2},
\end{equation*} we have
\begin{equation*}
\big\Vert T(f_1,f_2)\big\Vert_{L^{p,\infty}(\bbrn)}\lesssim M_0^{\theta_0}M_1^{\theta_1}M_2^{\theta_2}\Vert f_1\Vert_{L^{p_1}(\bbrn)}\Vert f_2\Vert_{L^{p_2}(\bbrn)}.
\end{equation*}
Moreover, if the points $(1/p_1^0,1/p_2^0)$, $(1/p_1^1,1/p_2^1)$, and $(1/p_1^2,1/p_2^2)$ form a non-trivial triangle in $\bbr^2$, then 
\begin{equation*}
\big\Vert T(f_1,f_2)\big\Vert_{L^{p}(\bbrn)}\lesssim M_0^{\theta_0}M_1^{\theta_1}M_2^{\theta_2}\Vert f_1\Vert_{L^{p_1}(\bbrn)}\Vert f_2\Vert_{L^{p_2}(\bbrn)}.
\end{equation*}

\end{lemma}

\hfill

\subsection{Compactly supported wavelets}

  Let $d$ be a positive integer.
Given two compactly supported functions $\psi_F$ and $\psi_M$ on $\bbr$, we define the function $\Psi_G$, defined on $\bbr^d$, by 
$$
\Psi_G(\xx):=\prod_{i=1}^d\psi_{G_i}(x_i), \qq \xx:=(x_1,\dots,x_d)\in \bbr^d
$$
where $G\in \mathcal I:=\{(G_1,G_2,\dots, G_d): G_{i}\in\{F,M\}\}$. 
Let $\Psi_{G,\nu}(\xx):=\Psi_{G}(\xx-\nu)$ be the translation of $\Psi_G$ by $\nu\in\bbz^d$ and 
 $\mathcal D_0:=\{\Psi_{G,\nu}:G\in\mathcal I,\ \nu\in\bbz^d\}.$
For $\la\in\bbn$ we define
$
\mathcal D_\la:=\{\Psi^{\la}_{ G,\nu}: \nu\in\bbz^d, G\in\mathcal I'\},
$
where $\Psi^{\la}_{ G,\nu}(\xx):=2^{\la d/2}\Psi_{G,\nu}(2^\la \xx-\nu)$, and $\mathcal I'=\mathcal I\setminus\{(F,F,\dots, F)\}$.
We denote $\mathcal D:=\cup_{\la=0}^\nf \mathcal D_\la$.

A classical result of Daubechies \cite{Da}  says that $\mathcal D$ is an orthonormal basis in $L^2(\bbr^d)$ for appropriate choices of $\psi_F$ and $\psi_M$.
\begin{lemma}\cite{Da}\label{daubechieslemma}
For any $N\in\mathbb N$, there exist $\psi_F$ and $\psi_M\in \mathscr{C}_c^N(\bbr)$ such that 
\begin{enumerate}
\item $\|\psi_F\|_{L^2(\bbr)}=\|\psi_M\|_{L^2(\bbr)}=1$,
\item $\int_\bbr x^\al \psi_M(x)dx=0$ for any $0\le \al\le N$,
\item $\mathcal D$ is an orthonormal basis in $L^2(\bbr^d)$.
\end{enumerate}
\end{lemma}
In this paper we will consider the case $d=2n$ and  write $\nnu=(\nu_1,\nu_2)\in\mathbb Z^n\times \mathbb Z^n$. 
We simplify our notations by writing $\Psi^{\la}_{ G,\nnu}(\yyy)=\om_{\nu_1}(y_1)\om_{\nu_2}(y_2)$ for $\yyy=(y_1,y_2)\in (\bbrn)^2$ and by using $\om_{\nnu}$ or just $\om$ for $\Psi_{G,\nnu}^{\la}$ when there is no confusion.

It is known that every $m\in L^2(\bbr^{2n})$ has the representation
$$
m=\sum_{\la=0}^{\infty}\sum_{G\in \II_{\la}}\sum_{\nnu\in (\bbzn)^2}\langle m,\Psi^{\la}_{ G,\nnu}\rangle \Psi^{\la}_{ G,\nnu},
$$
or simply
$m=\sum_{\la}\sum_{\om\in \mathcal D_\la}a_\om \om$ with $a_\om=\langle m,\om\rangle$. Here, $\II_{0}=\II$ and $\II_{\la}=\II'$ for $\la\in\bbn$.
Actually this representation holds for general spaces such as $L^q(\bbr^{2n})$ for $1<q<\infty$. We refer to \cite[Section 3]{Gr_He_Sl2020} for more details.

The following lemma slightly generalizes \cite[Theorem 1.1]{Gr_He_Sl2020}.
 \begin{lemma}\label{05031}
 Let $1\le r<4$ and $j\in\mathbb N$.
 Suppose that $m_0$ is supported in the annulus $\{\xxxi\in (\bbrn)^2: 2^{j-3}\le |\xxxi|\le 2^{j+3}\}$ and its wavelet decomposition $m_0=\sum_{\la}\sum_{\om\in\mathcal D_{\la}}a_\om \om$ satisfies
 \begin{enumerate}
\item $\|  \{a_\om\}_{\om\in\mathcal D_{\la}}\|_{\ell^r}\lesssim 2^{\la C(n,r)}$, 
\item $\| \{a_\om\}_{\om\in\mathcal D_{\la}} \|_{\ell^\nf}\lesssim  B 2^{-\la N}$ with  $B\le 1$ and $N\gg \tf r{4-r}C(n,r)$.
\end{enumerate}
Let $m_k:=m_0(2^{-k}\; \vec{\cdot}\;)$ for $k\in\mathbb Z$, and $m:=\sum_{k\in\mathbb Z}m_k$. 
Then the bilinear multiplier operators $T_m$ associated with $m$ satisfies 
\begin{equation}\label{e05032}
\big\Vert T_m(f_1,f_2)\big\|_{L^1(\rn)}\lesssim \max (jB^{1-\tf r4}, B)\|f_1\|_{L^2(\rn)}\|f_2\|_{L^2(\rn)}.
\end{equation}
\end{lemma}

The proof is essentially contained in \cite{Gr_He_Sl2020}, but for the sake of completeness, we include the proof in the appendix at the end of the paper.
We  also refer to \cite[Theorem 4]{Gr_He_Ho2018} for more related details.

\section{Proof of Theorem \ref{mainthm} : Reduction }\label{reduction}
In this section, following the idea of Duoandikoetxea and Rubio de Francia \cite{Du_Ru}, we shall reduce the proof of Theorem \ref{mainthm} to  operators with smooth kernels. 
We choose a Schwartz function $\Phi$ on $(\bbrn)^2$ such that its Fourier transform $\wh{\Phi}$ is supported in the annulus $\{\xxxi\in (\bbrn)^2: 1/2\le |\xxxi|\le 2\}$ and enjoys the property 
$\sum_{j\in\bbz}\wh{\Phi_j}(\xxxi)=1$ for $\xxxi\not= \0$ where $\wh{\Phi_j}(\xxxi):=\wh{\Phi}(\xxxi/2^j)$.
For $\ga\in\bbz$  let
 $$ K^{\ga}(\yyy):=\wh{\Phi}(2^{\ga}\yyy)K(\yyy), \quad \yyy\in (\bbrn)^2$$
 and then we observe that $K^\ga(\yyy)=2^{2 \ga n} K^0(2^\ga \yyy)$.
 For $\mu\in\bbz$ we define
 \begin{equation*}
K_{\mu}^{\ga}(\yyy):=\Phi_{{\mu}+\ga}\ast K^{\ga}(\yyy)=2^{\ga mn}\big( \Phi_{{\mu}}\ast K^{0}\big)(2^\ga \yyy).
\end{equation*}
Then we have
$$
\wh{K^\ga_\mu}(\xxxi)= \wh{\Phi}(2^{-(\mu+\ga)}\xxxi)\wh{K^0}(2^{-\ga}\xxxi)=\wh{K^0_\mu}(2^{-\ga} \xxxi),
$$
which implies that $\wh{K^\ga_\mu}$ is bounded uniformly in $\ga$ while they have almost disjoint supports, so it is natural to add them together as follows:
 $$K_{\mu}(\yyy):=\sum_{\ga\in \bbz}{K_{\mu}^{\ga}(\yyy)}.$$
 We define
 \begin{equation*}
 \LL_{\mu}\big(f_1,f_2\big)(x):=\int_{(\bbrn)^2}{K_{\mu}(\yyy)f_1(x-y_1)f_2(x-y_2)   } ~ d\yyy, \q x\in\bbrn
 \end{equation*} 
 and write
 \begin{equation*}
 \big\Vert \LL_{\Omega}(f_1,f_2)\big\Vert_{L^p}\le \Big(\sum_{\mu\in\bbz}\big\Vert \LL_{\mu}(f_1,f_2)\big\Vert_{L^p(\bbrn)}^{\min{(1,p)}}\Big)^{\frac{1}{\min{(1,p)}}}.
 \end{equation*}
 It is known in \cite{Gr_He_Ho2018, Paper1, Gr_He_Sl2020} that for all $1<p_1,p_2<\infty$, $1/2<p<\infty$ with $1/p=1/p_1+1/p_2$, and $1<q\le \infty$,
 \begin{equation}\label{basicest}
 \big\Vert \LL_{\mu}(f_1,f_2)\big\Vert_{L^p(\bbrn)}\lesssim \Vert \Omega\Vert_{L^q(\mathbb{S}^{2n-1})}\Vert f_1\Vert_{L^{p_1}(\bbrn)}\Vert f_2\Vert_{L^{p_2}(\bbrn)}
 \begin{cases}
 2^{(2n-\delta)\mu}, & \mu\ge 0\\
 2^{(1-\delta)\mu}, & \mu<0
 \end{cases} 
 \end{equation}
 where $0<\delta<1/q'$, and there exists an integer $\mu_0>0$ such that if $\mu\ge \mu_0$, then 
 \begin{equation}\label{essentialest}
 \big\Vert \LL_{\mu}(f_1,f_2)\big\Vert_{L^1(\bbrn)}\lesssim 2^{-\epsilon_0\mu}\Vert \Omega\Vert_{L^q(\mathbb{S}^{2n-1})}\Vert f_1\Vert_{L^2(\bbrn)}\Vert f_2\Vert_{L^2(\bbrn)}
 \end{equation}
 for some $\epsilon_0>0$.
 It follows from (\ref{basicest}) that
 $$\Big( \sum_{\mu<\mu_0} \big\Vert \LL_{\mu}(f_1,f_2)\big\Vert_{L^p(\bbrn)}^{\min{(1,p)}}\Big)^{\frac{1}{\min{(1,p)}}}\lesssim_{\mu_0} \Vert \Omega\Vert_{L^q(\mathbb{S}^{2n-1})}\Vert f_1\Vert_{L^{p_1}(\bbrn)}\Vert f_2\Vert_{L^{p_2}(\bbrn)}$$
 and thus it remains to show that for $\mu\ge \mu_0$ there exists a small constant $\delta_0>0$, possibly depending on $q,p_1,p_2$, such that
 \begin{equation}\label{mainidea}
 \big\Vert \LL_{\mu}(f_1,f_2)\big\Vert_{L^p(\bbrn)}\lesssim_{\delta_0}2^{-\delta_0 \mu} \Vert \Omega\Vert_{L^q(\mathbb{S}^{2n-1})}\Vert f_1\Vert_{L^{p_1}(\bbrn)}\Vert f_2\Vert_{L^{p_2}(\bbrn)},
 \end{equation}
 which clearly concludes
  $$\Big(\sum_{\mu\ge \mu_0} \big\Vert \LL_{\mu}(f_1,f_2)\big\Vert_{L^p(\bbrn)}^{\min{(1,p)}}\Big)^{\frac{1}{\min{(1,p)}}}\lesssim_{\mu_0,\delta_0} \Vert \Omega\Vert_{L^q(\mathbb{S}^{2n-1})}\Vert f_1\Vert_{L^{p_1}(\bbrn)}\Vert f_2\Vert_{L^{p_2}(\bbrn)}.$$
 
 The proof of (\ref{mainidea}) will be divided into three parts based on the region where the index $p$ is located;
\begin{equation*}
\begin{array}{ccc}
\text{Region $\mathrm{I}$} &: &1<p<\infty \\
\text{Region $\mathrm{II}$} &: & 1/2<p<1  \\
\text{Region $\mathrm{III}$} &: & p=1
\end{array}
\end{equation*}
 \begin{figure}[h]\label{fig2}
\begin{tikzpicture}
\path[fill=red!10] (0,0)--(0,2)--(2,0)--(0,0);
\draw [<->] (0,2.5)--(0,0)--(2.5,0);
\draw[dash pattern= { on 2pt off 1pt}] (0,2)--(2,0);
\node [below left] at (0,0) {\tiny$(0,0)$};
\node [below] at (2,0) {\tiny$(1,0)$};
\node [left] at (0,2) {\tiny$(0,1)$};
\node [below left] at (1.2,0.9) {\tiny$q>\frac{4}{3}$};
\node [right] at (2.5,0) {\tiny${\frac{1}{p_1}}$};
\node [above] at (0,2.5) {\tiny${\frac{1}{p_2}}$};
\node [below] at (1.3,-1) {Region $\mathrm{I}$};
\draw (0,0) circle [radius=0.04];
\draw (2,0) circle [radius=0.04];
\draw (0,2) circle [radius=0.04];
\path[fill=blue!10] (4.5,2)--(5.3,2)--(6.5,0.8)--(6.5,0)--(4.5,2);
\path[fill=blue!10] (6.5,2)--(6.5,0.8)--(5.3,2)--(6.5,2);
\draw [<->] (4.5,2.5)--(4.5,0)--(7,0);
\draw[dash pattern= { on 2pt off 1pt}] (4.5,2)--(6.5,2)--(6.5,0)--(4.5,2);
\draw[-] (5.3,2)--(6.5,0.8);
\draw[-] (6,1.3)--(7,1.6);
\node [right] at (7,1.6) {\tiny$q>\max{(\frac{4}{3},\frac{p}{2p-1})}$};
\node [right] at (8.15,1.2) {\tiny$\veq$};
\node [right] at (7.55,0.8) {\tiny$\max{(\frac{4}{3},\frac{p_0}{p_0-1})}$};

\node [left] at (4.5,2) {\tiny$(0,1)$};
\node [below left] at (4.5,0) {\tiny$(0,0)$};
\node [above] at (5.3,2) {\tiny$(\frac{1}{p_0},1)$};
\node [below] at (6.5,0) {\tiny$(1,0)$};
\node [right] at (6.4,0.8) {\tiny$(1,\frac{1}{p_0})$};
\node [right] at (6.5,2) {\tiny$(1,1)$};
\node [right] at (7,0) {\tiny${\frac{1}{p_1}}$};
\node [above] at (4.5,2.5) {\tiny${\frac{1}{p_2}}$};
\node [below] at (5.5,-1) {Region $\mathrm{II}$};

\draw [<->] (10,2.5)--(10,0)--(12.5,0);
\draw[-] (10,2)--(12,0);
\node [below left] at (10,0) {\tiny$(0,0)$};
\node [below] at (12,0) {\tiny$(1,0)$};
\node [left] at (10,2) {\tiny$(0,1)$};
\draw[-] (10.7,1.3)--(11.7,1.6);
\node [right] at (11.7,1.6) {\tiny $q>\frac{4}{3}$};
\node [below] at (11.3,-1) {Region $\mathrm{III}$};
\end{tikzpicture}
\caption{Regions $\mathrm{I}$, $\mathrm{II}$, and $\mathrm{III}$}
\end{figure}

Since the result in Region $\mathrm{III}$ can be obtained from  interpolation between the other two cases,
it is enough to deal only with indices $(1/p_1,1/p_2,1/p)$ in Regions $\mathrm{I}$ and $\mathrm{II}$.
This shall be done in the next two sections.

 \section{Proof of (\ref{mainidea}) in Region $\mathrm{I}$}\label{proofregion1}

As mentioned in  \cite[Lemma 6.4]{Gr_He_Sl2020}, using the argument in the proof of \cite[Corollary 4.1]{Du_Ru}, we can obtain
 \begin{equation}\label{1est}
 \big| \wh{K^{\ga}}(\xxxi)\big|\lesssim \Vert \Omega\Vert_{L^q(\mathbb{S}^{2n-1})}\min{\big( |2^{-\ga}\xxxi|, |2^{-\ga}\xxxi|^{-\delta}\big)}
 \end{equation}
 and
 \begin{equation}\label{2est}
 \big| \partial^{\alpha}\wh{K^{\ga}}(\xxxi)\big|\lesssim \Vert \Omega\Vert_{L^q} 2^{-\ga|\alpha|}\min{\big(1,|2^{-\ga}\xxxi|^{-\delta} \big)}
 \end{equation}
 for $q>1$, all multi-indices $\alpha$, and  $\delta$ satisfying $0<\delta<\frac{1}{2q'}$.
 
 The same estimates also hold for $\wh{K_{\mu}^{\ga}}$, whenever $\mu\ge \mu_0$, as follows:
 \begin{lemma}\label{lemmas}
 Let $\mu\ge \mu_0$, $\ga\in\bbz$, and $1<q\le \infty$. 
 Then we have
 \begin{equation*}
 \big| \wh{K_{\mu}^{\ga}}(\xxxi)\big|\lesssim \Vert \Omega\Vert_{L^q(\mathbb{S}^{2n-1})}\min{\big( |2^{-\ga}\xxxi|, |2^{-\ga}\xxxi|^{-\delta}\big)}
 \end{equation*}
 and
 \begin{equation}\label{2ndassertion}
 \big| \partial^{\alpha}\wh{K_{\mu}^{\ga}}(\xxxi)\big|\lesssim_{\mu_0,\alpha} \Vert \Omega\Vert_{L^q(\mathbb{S}^{2n-1})} 2^{-\ga|\alpha|}\min{\big(1,|2^{-\ga}\xxxi|^{-\delta} \big)}
 \end{equation}
  uniformly in $\mu\ge \mu_0$,
  for all multi-indices $\alpha$ and all $\delta$ satisfying $0<\delta<\frac{1}{2q'}$.
 \end{lemma}
 \begin{proof}
 We note that
 $$\wh{K_{\mu}^{\ga}}(\xxxi)=\wh{\Phi_{\mu+\ga}}(\xxxi)\wh{K^{\ga}}(\xxxi)$$
 and thus the first assertion immediately follows from the estimate (\ref{1est}).
 To verify the second one, we see that
 \begin{equation*}
 \big|\partial^{\alpha}\wh{K_{\mu}^{\ga}}(\xxxi)\big|\lesssim \sum_{\alpha_1+\alpha_2=\alpha} \big| \partial^{\alpha_1}\wh{\Phi_{\mu+\ga}}(\xxxi)\big| \big|\partial^{\alpha_2}\wh{K^{\ga}}(\xxxi)\big|.
 \end{equation*}
 Due to the support of $\wh{\Phi}$, we have
 $$\big|\partial^{\alpha_1}\wh{\Phi_{\mu+\ga}}(\xxxi)\big|=2^{-|\alpha_1|(\mu+\ga)}\big|\partial^{\alpha_1}\wh{\Phi}(\xxxi/2^{\mu+\ga})\big|\lesssim_{\mu_0}2^{-|\alpha_1|\ga}\chi_{|\xxxi|\sim 2^{\mu+\ga}}(\xxxi)$$
 and the estimates (\ref{1est}) and (\ref{2est}) imply that
 \begin{equation*}
 \big|\partial^{\alpha_2}\wh{K^{\ga}}(\xxxi)\big|\lesssim \Vert \Omega\Vert_{L^q(\mathbb{S}^{2n-1})} 2^{-\ga|\alpha_2|}
 \begin{cases}
  \min{\big(|2^{-\ga}\xxxi|, | 2^{-\ga}\xxxi |^{-\delta} \big)}, & \alpha_2= 0\\
 \min{\big(1, | 2^{-\ga}\xxxi|^{-\delta} \big)}, & \alpha_2\not= 0
 \end{cases}.
 \end{equation*}
 If $|\xxxi|\sim 2^{\ga+\mu}$ for $\mu\ge \mu_0$, then 
 $$\min{\big(1,|2^{-\ga}\xxxi|^{-\delta} \big)}=|2^{-\ga}\xxxi|^{-\delta}\le \min{\big(|2^{-\ga}\xxxi|, |2^{-\ga}\xxxi|^{-\delta} \big)}$$
 and finally, combining all together, we obtain (\ref{2ndassertion}).
 \end{proof}

We now generalize the estimate (\ref{essentialest}) by using Lemma \ref{lemmas}.
   \begin{proposition}\label{05032}
 Let $2\le p_1,p_2\le\nf$ and $1\le p\le 2$ with $1/p=1/p_1+1/p_2$. 
 Suppose that $4/3<q\le \infty$ and $\mu\ge \mu_0$.
 Then we have  
\begin{equation}\label{2i2}
 \|\mathcal L_\mu(f_1,f_2)\|_{L^p(\bbrn)}\lesssim  2^{-\mu\ep_0}\|\Om\|_{L^q(\mathbb S^{2n-1})}\|f_1\|_{L^{p_1}(\rn)}\|f_2\|_{L^{p_2}(\rn)}.
 \end{equation}
 for some $\ep_0>0$ .
 \end{proposition}

The following two propositions are end-point estimates for $\LL_{\mu}$, which will be finally interpolated with (\ref{2i2}).
 \begin{proposition}\label{1i1}
 Let $\delta>0$ and $\mu\ge \mu_0$. Suppose that $4/3<q\le \infty$.
 Then we have
 \begin{equation*}
 \big\Vert \LL_{\mu}(f_1,f_2)\big\Vert_{L^{1,\infty}(\bbrn)}\lesssim_{\delta} 2^{\delta \mu}\Vert \Omega\Vert_{L^{q}(\mathbb{S}^{2n-1})}\Vert f_1\Vert_{L^{1}(\bbrn)}\Vert f_2\Vert_{L^{\infty}(\bbrn)}
 \end{equation*}
 and
  \begin{equation*}
 \big\Vert \LL_{\mu}(f_1,f_2)\big\Vert_{L^{1,\infty}(\bbrn)}\lesssim_{\delta} 2^{\delta \mu}\Vert \Omega\Vert_{L^{q}(\mathbb{S}^{2n-1})}\Vert f_1\Vert_{L^{\infty}(\bbrn)}\Vert f_2\Vert_{L^{1}(\bbrn)}.
 \end{equation*}
 \end{proposition}
 
\begin{proposition}\label{iii}
Let $\delta>0$ and $\mu\ge \mu_0$. Suppose that $4/3<q\le \infty$.
 Then we have
   \begin{equation*}
 \big\Vert \LL_{\mu}(f_1,f_2)\big\Vert_{BMO(\bbrn)}\lesssim_{\delta} 2^{\delta \mu}\Vert \Omega\Vert_{L^{q}(\mathbb{S}^{2n-1})}\Vert f_1\Vert_{L^{\infty}(\bbrn)}\Vert f_2\Vert_{L^{\infty}(\bbrn)}.
 \end{equation*}
\end{proposition}

Then (\ref{mainidea}) follows from interpolating  
$$\begin{array}{cl}
L^2\times L^2\to L^1 &\text{  in (\ref{essentialest})},\\
L^2\times L^{\infty}\to L^2,\q L^{\infty}\times L^{2}\to L^2 &\text{  in Proposition \ref{05032}},\\
L^1\times L^{\infty}\to L^{1,\infty},\q L^{\infty}\times L^1\to L^{1,\infty} &\text{  in Proposition \ref{1i1}},\\
L^{\infty}\times L^{\infty}\to BMO &\text{  in Proposition \ref{iii}}
\end{array}$$
because we can fix $\epsilon_0>0$ in the $L^2\times L^2\to L^1$, $L^2\times L^{\infty}\to L^2,\q L^{\infty}\times L^{2}\to L^2$ estimates and choose $\delta>0$ sufficiently small, depending on $p_1,p_2$, in the other estimates.
To be specific, we first obtain
\begin{equation}\label{p1infty}
 \big\Vert \LL_{\mu}(f_1,f_2)\big\Vert_{L^{p_1}(\bbrn)}\lesssim_{\delta} 2^{-\delta_0 \mu}\Vert \Omega\Vert_{L^{q}(\mathbb{S}^{2n-1})}\Vert f_1\Vert_{L^{p_1}(\bbrn)}\Vert f_2\Vert_{L^{\infty}(\bbrn)}
 \end{equation}
 for some $\delta_0>0$ where $1<p_1<\infty$, by choosing $\de>0$ in Proposition~\ref{1i1} and Proposition~\ref{iii} small enough, and
 by using (linear) Marcinkiewicz interpolation method 
 between the boundedness results at $(1/2,0)$ and $(1,0)$, and by applying either (\ref{realinter}) or (\ref{complexinter}) to interpolate the results at $(1/2,0)$ and $(0,0)$.
A similar argument gives \eqref{mainidea} for all $(1/p_1,1/p_2)$ on the boundary of Region I except the points $(0,0)$, $(0,1)$, and $(1,0)$,  where $L^p(\bbrn)$ in (\ref{mainidea}) should be replaced by $L^{p,\infty}(\bbrn)$ if $p=1$.
Applying Lemma \ref{interpollemma} again, we obtain \eqref{mainidea} for all $(1/p_1,1/p_2)$ in the interior of Region I.

 This finishes the proof of (\ref{mainidea}) in Region $\mathrm{I}$.

\hfill

We now prove the above three propositions.

\bpf[Proof of Proposition~\ref{05032}]
  We may assume that $4/3<q\le 2$ since  $\mathbb S^{2n-1}$ is compact.
 Moreover, as the case $(p_1,p_2,p)=(2,2,1)$ has been already established in \eqref{essentialest}, by multilinear interpolation, it suffices to verify \eqref{2i2} for $(p_1,p_2,p)=(2,\nf,2)$ and $(\nf,2,2)$. We will focus on the case when $(p_1,p_2,p)=(2,\nf,2)$ since the other case follows by symmetry.
Therefore, matters reduce to the inequality
\begin{equation*}
\big\Vert \LL_{\mu}(f_1,f_2)\big\Vert_{L^{2}(\bbrn)}\lesssim 2^{-\mu \epsilon_0}\Vert \Omega\Vert_{L^q(\mathbb{S}^{2n-1})}\Vert f_1\Vert_{L^{2}(\bbrn)}\Vert f_2\Vert_{L^{\infty}(\bbrn)},
\end{equation*}
 which is actually equivalent to
 \begin{equation}\label{equivest}
\big\Vert \LL_{\mu}^{*2}(f_1,f_2)\big\Vert_{L^{1}(\bbrn)}\lesssim 2^{-\mu \epsilon_0}\Vert \Omega\Vert_{L^q(\mathbb{S}^{2n-1})}\Vert f_1\Vert_{L^{2}(\bbrn)}\Vert f_2\Vert_{L^2(\bbrn)},
\end{equation}
where $\LL_{\mu}^{*2}$ is the second transpose of $\LL_{\mu}$, defined via
$$\langle \LL_{\mu}^{*2}(f_1,f_2),h\rangle=\langle \LL_{\mu}(f_1,h),f_2\rangle$$
 for all Schwartz functions $h$ on $\bbrn$.
 We note that that $\LL_{\mu}^{*2}$ would be the bilinear multiplier operator $T_{M_{\mu}}$ associated with 
 $$M_{\mu}(\xi_1,\xi_2):=\wh{K^0_\mu}(\xi_1,-(\xi_1+\xi_2)).$$
 We observe that $\big|\big(\xi_1,-(\xi_1+\xi_2)\big)\big|\sim |(\xi_1,\xi_2)|$ and thus $M_{\mu}$ is supported in $B(0,2^{\mu+3})\setminus B(0,2^{\mu-3})$.
 Now let $\om$ be the wavelets that appeared in Lemma \ref{05031} and we define
 $$a_{\om}:=\langle M_{\mu},\om\rangle.$$
 Then we claim that
 \begin{equation}\label{claim1}
 \big\| \{a_\om\}_{\om\in\mathcal D_{\la}}\big\|_{\ell^{q'}}\lesssim 2^{\la n( 2/{q'}-1)} \|\Om\|_{L^q(\mathbb{S}^{2n-1})} 
 \end{equation}
 and
\begin{equation}\label{claim2}
 \big\| \{a_\om\}_{\om\in\mathcal D_{\la}}\big\|_{\ell^\nf}\lesssim 2^{-\mu\de} 2^{-\la N} \|\Om\|_{L^q(\mathbb S^{2n-1})} 
\end{equation}
where $N$ is the number of vanishing moment of $\psi_M$ in Lemma \ref{daubechieslemma},  which can be chosen arbitrarily large as we need. 
Those claims verify the assumptions of Lemma \ref{05031} with $B=2^{-\mu\de}$ and finally (\ref{equivest}) follows from \eqref{e05032}.
 
 Now let us prove the above two claims (\ref{claim1}) and (\ref{claim2}).

It was proved in \cite[(11)]{Gr_He_Sl2020} that
$$ \big\| \{a_\om\}_{\om\in\mathcal D_{\la}}\big\|_{\ell^{q'}}\lesssim 2^{\la n(2/q'-1)}\Vert  M_{\mu}\Vert_{L^{q'}((\bbrn)^2)}$$
and for $4/3<q\le 2$ we have
$$\Vert  M_{\mu}\Vert_{L^{q'}((\bbrn)^2)}= \big\Vert \wh{K_{\mu}^0}\big\Vert_{L^{q'}((\bbrn)^2)}\lesssim \Vert K_{\mu}^0\Vert_{L^q((\bbrn)^2)}\lesssim \Vert \Omega\Vert_{L^q(\mathbb{S}^{2n-1})}$$
where the Hausdorff-Young inequality is applied. This proves (\ref{claim1}).

To show (\ref{claim2}), we apply Lemma \ref{lemmas} and obtain
$$\big| \p^\al \wh{K_{\mu}^{0}}(\xxxi)\big|\lesssim_\al \|\Om\|_{L^q(\mathbb S^{2n-1})}\min{(1, |\xxxi|^{-\delta})}\sim 2^{-\delta \mu} \|\Om\|_{L^q(\mathbb S^{2n-1})} $$
 for $\mu\ge \mu_0$, which further implies that 
$$\big| \p^\al M_{\mu}(\xxxi)\big|\lesssim_\al 2^{-\mu \delta}\|\Om\|_{L^q(\mathbb S^{2n-1})}.$$
Then (\ref{claim2}) follows from  \cite[Lemma 2.1]{Gr_He_Sl2020}.
 \epf

 \begin{proof}[Proof of Proposition \ref{1i1}]
 We are only concerned with the first estimate as a symmetric argument is applicable to the other one. 
 Without loss of generality, we may assume $\Vert f_1\Vert_{L^1(\bbrn)}=\Vert f_2\Vert_{L^{\infty}(\bbrn)}=1$ and $\Vert \Omega\Vert_{L^q(\mathbb{S}^{2n-1})}=1$.
 Then it suffices to show that for all $\delta>0$ and $t>0$
 \begin{equation}\label{weakgoal}
 \Big| \Big\{x\in\bbrn : \big| \LL_{\mu}\big(f_1,f_2\big)(x)\big|>t \Big\}\Big|\lesssim_{\delta}2^{\delta \mu}\frac{1}{t}.
 \end{equation}
 We shall use the Calder\'on-Zygmund decomposition of $f_1$ at height $t$. Then $f_1$ can be expressed as
 $$f_1=g_1+\sum_{Q\in \mathcal{A}}b_{1,Q}$$
 where $\mathcal{A}$ is a subset of disjoint dyadic cubes, $\big| \bigcup_{Q\in\mathcal{A}} Q\big|\lesssim \frac{1}{t}$, $\supp (b_{1,Q})\subset Q$, $\int{b_{1,Q}(y)}dy=0$, $\Vert b_{1,Q}\Vert_{L^1(\bbrn)}\lesssim t |Q|$, and $\Vert g_1\Vert_{L^r(\bbrn)}\lesssim t^{1-1/r}$ for all $1\le r\le \infty$.
 
 The left-hand side of (\ref{weakgoal}) is less than
 \begin{equation*}
 \Big| \Big\{x\in\bbrn : \big| \LL_{\mu}\big(g_1,f_2\big)(x)\big|>\frac{t}{2} \Big\}\Big|+\Big| \Big\{x\in\bbrn : \Big| \LL_{\mu}\Big(\sum_{Q\in\mathcal{A}}b_{1,Q},f_2\Big)(x)\Big|>\frac{t}{2} \Big\}\Big|.
 \end{equation*}
 Using Chebyshev's inequality, the first term is clearly dominated by
 $$\frac{1}{t^2}\big\Vert \LL_{\mu}(g_1,f_2)\big\Vert_{L^2(\bbrn)}^2\lesssim \frac{1}{t^2}\Vert \Omega\Vert_{L^q(\mathbb{S}^{2n-1})}^2\Vert g_1\Vert_{L^2(\bbrn)}^2\Vert f_2\Vert_{L^{\infty}(\bbrn)}^2\lesssim\frac{1}{t}$$
 where the first inequality follows from the $L^2\times L^{\infty}\to L^2$ estimate in (\ref{2i2}).
 Moreover, the remaining term is estimated by the sum of $\big| \cup_{Q\in\mathcal{A}}Q^*\big|$
 and 
 \begin{equation*}
\Gamma_{\mu}:= \Big| \Big\{x\in \Big( \bigcup_{Q\in\mathcal{A}}Q^*\Big)^c : \Big| \LL_{\mu}\Big(\sum_{Q\in\mathcal{A}}b_{1,Q},f_2\Big)(x)\Big|>\frac{t}{2} \Big\}\Big|
 \end{equation*}
 where we recall that $Q^*$ is the concentric dilation of $Q$ with $\ell(Q^*)=10^2\sqrt{n}\ell(Q)$.
 Since $\big| \cup_{Q\in\mathcal{A}}Q^*\big|\lesssim \big| \cup_{Q\in\mathcal{A}}Q\big|\lesssim \frac{1}{t}$, it is sufficient to show that 
 \begin{equation}\label{weakfinalgoal}
 \Gamma_{\mu}\lesssim_{\delta}2^{\delta \mu}\frac{1}{t}.
 \end{equation}
 We apply Chebyshev's inequality to deduce
 \begin{align*}
 \Gamma_{\mu}&\le \frac{1}{t}\Big\Vert \LL_{\mu}\Big(\sum_{Q\in\mathcal{A}}b_{1,Q},f_2 \Big) \Big\Vert_{L^1((\cup_{Q\in\mathcal{A}}Q*)^c)}\le \frac{1}{t}\sum_{Q\in\mathcal{A}}\big\Vert \LL_{\mu}(b_{1,Q},f_2)\big\Vert_{L^1((Q^*)^c)}\\
 &\le \frac{1}{t}\sum_{Q\in\mathcal{A}}\sum_{\ga\in\bbz}\int_{(Q^*)^c}{    \Big| T_{K_{\mu}^{\ga}}\big( b_{1,Q},f_2\big)(x) \Big|       }dx\\
 &\le \frac{1}{t}\sum_{Q\in\mathcal{A}}\sum_{\ga: 2^{\ga}\ell(Q)<1}\cdots+\frac{1}{t}\sum_{Q\in\mathcal{A}}\sum_{\ga: 2^{\ga}\ell(Q)\ge 1}\cdots=:\Gamma_{\mu}^1+\Gamma_{\mu}^2
 \end{align*}
 where $T_{K_{\mu}^{\ga}}$ is the bilinear operator associated with the kernel $K_{\mu}^{\ga}$ so that
 \begin{equation}\label{tkmuga}
  T_{K_{\mu}^{\ga}}\big( b_{1,Q},f_2\big)(x)=\int K_{\mu}^{\ga}(x-y_1,x-y_2)b_{1,Q}(y_1)f_2(y_2)d\yyy.
  \end{equation}

 To estimate $\Gamma_{\mu}^1$, we use the vanishing moment condition of $b_{1,Q}$ and then obtain
 \begin{align}\label{gamma1est}
 &\int_{(Q^*)^c}\big| T_{K_{\mu}^{\ga}}\big( b_{1,Q},f_2\big)(x)\big|dx\nonumber\\
 &\le \int_{(Q^*)^c}\int_{(\bbrn)^2} \big| K_{\mu}^{\ga}(x-y_1,x-y_2)-K_{\mu}^{\ga}(x-c_Q,x-y_2)\big| |b_{1,Q}(y_1)|  |f_2(y_2)|      d\yyy dx\nonumber\\
 &\lesssim \int_{(Q^*)^c}\int_{\yyy\in(\bbrn)^2} \int_{\zzz\in (\bbrn)^2:|\zzz|\sim 2^{-\ga}}2^{2\ga n}\big|\Omega(\zzz\;') \big|\big| \Phi_{\mu+\ga}(x-y_1-z_1,x-y_2-z_2)\nonumber\\
 &\qq \qq \qq\qq\qq \qq -\Phi_{\mu+\ga}(x-c_Q-z_1,x-y_2-z_2)\big| |b_{1,Q}(y_1)|d\zzz d\yyy dx\nonumber\\
 &=\int_{y_1\in Q}|b_{1,Q}(y_1)|{\int_{\zzz\in (\bbrn)^2:|\zzz|\sim 2^{-\ga}} { 2^{2\ga n}|\Omega(\zzz\;')| \Lambda_{\mu+\ga}(y_1,c_Q,z_1)       }      d\zzz  }dy_1
 \end{align}
 where
 $$\Lambda_{\mu+\ga}(y_1,c_Q,z_1):=\int_{(x,y_2)\in (\bbrn)^2} \big|\Phi_{\mu+\ga}(x-y_1-z_1,y_2)-\Phi_{\mu+\ga}(x-c_Q-z_1,y_2) \big|     dx dy_2.$$
 Here, we used a change of variables $x-y_2-z_2\mapsto y_2$ in the identity.
 We first observe that
 $$\Lambda_{\mu+\ga}(y_1,c_Q,z_1)\le 2\int_{(x,y_2)\in (\bbrn)^2}\big| \Phi_{\mu+\ga}(x,y_2)\big|dxdy_2\lesssim 1.$$
 Furthermore, for $y_1\in Q$
 \begin{align*}
 &\Lambda_{\mu+\ga}(y_1,c_Q,z_1)\\
 &\lesssim 2^{\mu+\ga}|y_1-c_Q|\int_{(\bbrn)^2}  \Big( \int_0^1\frac{2^{2(\mu+\ga)n}}{(1+2^{\mu+\ga}|x-sy_1-(1-s)c_Q|+2^{\mu+\ga}|y_2|)^{2n+1}}ds\Big) dx dy_2\\
 &\lesssim 2^{\mu+\ga}\ell(Q).
 \end{align*}
Averaging the two estimates, we obtain 
\begin{equation}\label{lambdaest}
\Lambda_{\mu+\ga}(y_1,c_Q,z_1)\lesssim_{\delta} 2^{\delta\mu} \big(2^{\ga}\ell(Q)\big)^{\delta}.
\end{equation}
 By plugging (\ref{lambdaest}) into (\ref{gamma1est}), 
 \begin{align*}
 \int_{(Q^*)^c}\big| T_{K_{\mu}^{\ga}}\big( b_{1,Q},f_2\big)(x)\big|dx &\lesssim 2^{\delta \mu}\big( 2^{\ga}\ell(Q)\big)^{\delta}\Vert \Omega\Vert_{L^1(\mathbb{S}^{2n-1})}\Vert b_{1,Q}\Vert_{L^1(\bbrn)}\\
 &\le 2^{\delta\mu}\big( 2^{\ga}\ell(Q)\big)^{\delta}\Vert b_{1,Q}\Vert_{L^1(\bbrn)}
 \end{align*}
 and accordingly,
 \begin{equation*}
 \Gamma_{\mu}^1\lesssim 2^{\delta\mu}\frac{1}{t}\sum_{Q\in\mathcal{A}}\Vert b_{1,Q}\Vert_{L^1(\bbrn)}\sum_{\ga:2^{\ga}\ell(Q)<1}\big( 2^{\ga}\ell(Q)\big)^{\delta}\lesssim_{\delta} 2^{\delta\mu}\frac{1}{t} \sum_{Q\in\mathcal{A}} t|Q|\lesssim 2^{\delta\mu}\frac{1}{t}.
 \end{equation*}
 
 Now let us consider $\Gamma_{\mu}^2$. Assume $2^{\ga}\ell(Q)\ge 1$.
 Then
 \begin{align}\label{gamma2est}
& \int_{(Q^*)^c}{\big|T_{K_{\mu}^{\ga}}\big(b_{1,Q},f_2 \big)(x) \big|}dx \nonumber\\
&\lesssim \int_{y_1\in Q}|b_{1,Q}(y_1)|{\int_{\zzz\in (\bbrn)^2:|\zzz|\sim 2^{-\ga}} { 2^{2\ga n}|\Omega(\zzz\;')| \Theta^Q_{\mu+\ga}(y_1,z_1)       }      d\zzz  }dy_1
 \end{align}
 where $$\Theta_{\mu+\ga}^Q(y_1,z_1):= \int_{x\in (Q^*)^c}\int_{y_2\in\bbrn} \big| \Phi_{\mu+\ga}(x-y_1-z_1,y_2)\big|dxdy_2.$$
 Note that
 \begin{align*}
 \Theta_{\mu+\ga}^Q(y_1,z_1)&\lesssim \int_{(Q^*)^c} \frac{2^{(\mu+\ga)n}}{(1+2^{\mu+\ga}|x-y_1-z_1|)^{n+1}}dx \int_{\bbrn}\frac{2^{(\mu+\ga)n}}{(1+2^{\mu+\ga}|y_2|)^{n+1}}dy_2.
 \end{align*}
 It is clear that the second integral is dominated by a constant, and for the other integral we notice that for $x\in (Q^*)^c$, $y_1\in Q$, and $|z_1|\le 2^{-\ga+1}$ with $2^{\ga}\ell(Q)\ge 1$,
 $$|x-y_1-z_1|\gtrsim |x-c_Q|,$$
  which deduces
  \begin{equation}\label{thetaest}
   \Theta_{\mu+\ga}^Q(y_1,z_1)\lesssim \int_{(Q^*)^c}\frac{2^{(\mu+\ga)n}}{(1+2^{\mu+\ga}|x-c_Q|)^{n+1}}dx\lesssim 2^{-\mu} \big(2^{\ga}\ell(Q)\big)^{-1}\le \big(2^{\ga}\ell(Q)\big)^{-1}
  \end{equation} for $\mu\ge \mu_0$.
 Combining (\ref{gamma2est}) and (\ref{thetaest}), we obtain
 $$\int_{(Q^*)^c}{\big|T_{K_{\mu}^{\ga}}\big(b_{1,Q},f_2 \big)(x) \big|}dx\lesssim \big(2^{\ga}\ell(Q)\big)^{-1} \Vert b_{1,Q}\Vert_{L^1(\bbrn)},$$
 which finally proves that
 \begin{equation*}
 \Gamma_{\mu}^2\lesssim \frac{1}{t}\sum_{Q\in\mathcal{A}}\Vert b_{1,Q}\Vert_{L^1(\bbrn)}\sum_{\ga: 2^{\ga}\ell(Q)\ge 1}\big(2^{\ga}\ell(Q)\big)^{-1} \lesssim \frac{1}{t}.
 \end{equation*}
 
 This completes the proof of (\ref{weakfinalgoal}).
 \end{proof}

 \begin{proof}[Proof of Proposition \ref{iii}]
 Without loss of generality, we may assume $\Vert f_1\Vert_{L^{\infty}(\bbrn)}=\Vert f_2\Vert_{L^{\infty}(\bbrn)}=\Vert \Omega\Vert_{L^q(\mathbb{S}^{2n-1})}=1$.
 Let $T_{K_{\mu}^{\ga}}$ be the bilinear operator associated with the kernel $K_{\mu}^{\ga}$ as in (\ref{tkmuga}).
 By the definition in (\ref{bmodef}), we have
 \begin{align*}
 \big\Vert \LL_{\mu}(f_1,f_2)\big\Vert_{BMO(\bbrn)}&=\sup_{Q}\inf_{a\in\bbc}\frac{1}{|Q|}\int_Q{\big|\LL_{\mu}\big(f_1,f_2\big)(x)-a \big|}dx.
 \end{align*}
 Since $$\big|\LL_{\mu}\big(f_1,f_2\big)(x) -a\big|   \le \big|\LL_{\mu}\big(\chi_{Q^*}f_1,f_2\big)(x) \big|+\big|\LL_{\mu}\big(\chi_{(Q^*)^c}f_1,f_2\big)(x)-a \big|,$$
 the $BMO$ norm in the preceding expression is less than the sum of
 $$\II_1:=\sup_{Q}\frac{1}{|Q|}\int_Q\big| \LL_{\mu}\big(\chi_{Q^*}f_1,f_2\big)(x)\big|dx$$
 and
 $$\II_2:=\sup_{Q}\inf_{a\in\bbc}\frac{1}{|Q|}\int_Q{\big|\LL_{\mu}\big(\chi_{(Q^*)^c}f_1,f_2\big)(x)-a \big|}dx.$$
 
From the $L^2\times L^{\infty}\to L^2$ estimate in (\ref{2i2}), it follows that
 $$\II_1\le \sup_{Q}\frac{1}{|Q|^{1/2}}\big\Vert \LL_{\mu}({\chi_{Q^*}f_1,f_2)\big\Vert_{L^2(\bbrn)}}\lesssim \sup_{Q}\frac{1}{|Q|^{1/2}}\Vert f_1\Vert_{L^2(Q^*)}\lesssim 1.$$
 
 On the other hand, we have
 \begin{equation*}
 \II_2\le \sup_{Q}\frac{1}{|Q|}\int_Q  \big| \LL_{\mu}\big(\chi_{(Q^*)^c}f_1,f_2\big)(x)- \LL_{\mu}\big(\chi_{(Q^*)^c}f_1,f_2\big)(c_Q)     \big|  dx
 \end{equation*}
 and
 \begin{align*}
 & \big| \LL_{\mu}\big(\chi_{(Q^*)^c}f_1,f_2\big)(x)- \LL_{\mu}\big(\chi_{(Q^*)^c}f_1,f_2\big)(c_Q)     \big| \\
 &\le \sum_{\ga\in\bbz}\int_{\yyy\in(\bbrn)^2}\big|K_{\mu}^{\ga}(x-y_1,x-y_2)-K_{\mu}^{\ga}(c_Q-y_1,c_Q-y_2) \big|\chi_{(Q^*)^c}(y_1)d\yyy\\
 &\le \sum_{\ga:2^{\ga}\ell(Q)<1}\cdots +\sum_{\ga:2^{\ga}\ell(Q)\ge 1}\cdots =:\JJ_1+\JJ_2.
 \end{align*}
 We note that
 \begin{align*}
 &\int_{\yyy\in(\bbrn)^2}\big|K_{\mu}^{\ga}(x-y_1,x-y_2)-K_{\mu}^{\ga}(c_Q-y_1,c_Q-y_2) \big|\chi_{(Q^*)^c}(y_1)d\yyy\\
 &\lesssim \int_{\zzz\in (\bbrn)^2:|\zzz|\sim 2^{-\ga}} 2^{2\ga n}|\Omega(\zzz\;')|   \;  \Xi_{\mu+\ga}(x,c_Q,z_1)    \;     d\zzz
 \end{align*}
 where \begin{align*}
 &\Xi_{\mu+\ga}(x,c_Q,z_1)\\
 &:= \int_{(\bbrn)^2}\big| \Phi_{\mu+\ga}(x-y_1-z_1,x-y_2) -\Phi_{\mu+\ga}(c_Q-y_1-z_1,c_Q-y_2)\big|\chi_{(Q^*)^c}(y_1)d\yyy.
 \end{align*}
 
 For the estimation of $\JJ_1$, we write
 \begin{align*}
 &\Xi_{\mu+\ga}(x,c_Q,z_1)\\
 &\le  \int_{ (\bbrn)^2}\big| \Phi_{\mu+\ga}(x-y_1-z_1,x-y_2) -\Phi_{\mu+\ga}(c_Q-y_1-z_1,x-y_2)\big|d\yyy\\
 &\q+ \int_{(\bbrn)^2}\big| \Phi_{\mu+\ga}(c_Q-y_1-z_1,x-y_2) -\Phi_{\mu+\ga}(c_Q-y_1-z_1,c_Q-y_2)\big|d\yyy\\
 &=\int_{ (\bbrn)^2}\big| \Phi_{\mu+\ga}(x-y_1,y_2)-\Phi_{\mu+\ga}(c_Q-y_1,y_2)\big| d\yyy\\
 &\qq +\int_{ (\bbrn)^2}\big| \Phi_{\mu+\ga}(y_1,x-y_2)-\Phi_{\mu+\ga}(y_1,c_Q-y_2)\big| d\yyy,
 \end{align*} using a change of variables. Since $x\in Q$, the preceding expression is controlled by
 $2^{\delta \mu} \big( 2^{\ga}\ell(Q)\big)^{\delta},$
 using the argument that led to (\ref{lambdaest}). This shows that
 \begin{equation*}
 \JJ_1\lesssim_{\delta} 2^{\delta \mu} \Vert \Omega\Vert_{L^1(\mathbb{S}^{2n-1})}\sum_{\ga:2^{\ga}\ell(Q)<1} \big( 2^{\ga}\ell(Q)\big)^{\delta} \lesssim_{\delta} 2^{\delta \mu}.
 \end{equation*}
 
 Moreover, we have
 \begin{align*}
 \Xi_{\mu+\ga}(x,c_Q,z_1)&\le \int_{(\bbrn)^2}\big| \Phi_{\mu+\ga}(x-y_1-z_1,y_2)\big|\chi_{(Q^*)^c}(y_1)d\yyy\\
  &\qq+\int_{(\bbrn)^2}\big| \Phi_{\mu+\ga}(c_Q-y_1-z_1,y_2)\big|\chi_{(Q^*)^c}(y_1)d\yyy\\
  &\lesssim \int_{(Q^*)^c} \frac{2^{(\mu+\ga)n}}{(1+2^{\mu+\ga}|x-y_1-z_1|)^{n+1}}dy_1\\
  &\qq+\int_{(Q^*)^c} \frac{2^{(\mu+\ga)n}}{(1+2^{\mu+\ga}|c_Q-y_1-z_1|)^{n+1}}dy_1.
 \end{align*}
 If $2^{\ga}\ell(Q)\ge 1$, then
 $$|x-y_1-z_1|, |c_Q-y_1-z_1|\gtrsim |y_1-c_Q|$$
  for $x\in Q$, $y_1\in (Q^*)^c$, and $|z_1|\le 2^{-\ga+1}$.
 This yields that
 $$ \Xi_{\mu+\ga}(x,c_Q,z_1)\lesssim 2^{-\mu}\big( 2^{\ga}\ell(Q)\big)^{-1}\le \big( 2^{\ga}\ell(Q)\big)^{-1}$$ for $\mu\ge \mu_0$
 and thus
 $$\JJ_2\lesssim \Vert \Omega\Vert_{L^1(\mathbb{S}^{2n-1})}\sum_{\ga:2^{\ga}\ell(Q)\ge 1}\big( 2^{\ga}\ell(Q)\big)^{-1}\lesssim 1.$$
 
 Finally, we arrive at the inquality
 $$\II_2\lesssim_{\delta} 2^{\delta \mu}$$
   for all $\delta>0$, which completes the proof of Propisition \ref{iii}.
 \end{proof}

 \section{Proof of (\ref{mainidea}) in Region $\mathrm{II}$}
 
 In this section, we consider $1<p_1,p_2<\infty$ and $1/2<p<1$ satisfying $1/p=1/p_1+1/p_2$.

 Choose $1<p_0<\infty$ so that 
 \begin{equation*}
 1+\frac{1}{p_0}=\frac{1}{p} \qq \Big( \text{ that is, }~p=\frac{p_0}{p_0+1}~\Big).
 \end{equation*} 
 Then we shall show 
 two 
  end-point estimates $L^{1}\times L^{p_0}\to L^{p,\infty}$ and $L^{p_0}\times L^1\to L^{p,\infty}$ for which the Calder\'on-Zygmund decomposition is applicable as in Proposition \ref{1i1}. 
 We note that $\frac{p}{2p-1}=\frac{p_0}{p_0-1}$ in this case.
 \begin{proposition}\label{pip}
 Let $\delta>0$, $\mu\ge \mu_0$, and $1<p_0<\infty$.
 Suppose that $$\max{\Big(\frac{4}{3},\frac{p_0}{p_0-1}\Big)}<q\le \infty.$$
 Then we have
 \begin{equation}\label{propo311}
 \big\Vert \LL_{\mu}(f_1,f_2)\big\Vert_{L^{\frac{p_0}{p_0+1},\infty}(\bbrn)}\lesssim_{\delta} 2^{\delta \mu}\Vert \Omega\Vert_{L^{q}(\mathbb{S}^{2n-1})}\Vert f_1\Vert_{L^{1}(\bbrn)}\Vert f_2\Vert_{L^{p_0}(\bbrn)}
 \end{equation}
 and
  \begin{equation}\label{propo312}
 \big\Vert \LL_{\mu}(f_1,f_2)\big\Vert_{L^{\frac{p_0}{p_0+1},\infty}(\bbrn)}\lesssim_{\delta} 2^{\delta \mu}\Vert \Omega\Vert_{L^{q}(\mathbb{S}^{2n-1})}\Vert f_1\Vert_{L^{p_0}(\bbrn)}\Vert f_2\Vert_{L^{1}(\bbrn)}.
 \end{equation}
 \end{proposition}
\begin{figure}[h]
\begin{tikzpicture}
\draw [<->] (0,5)--(0,0)--(5,0);
\draw[dash pattern= { on 2pt off 1pt}] (0,4)--(4,4)--(4,0);
\draw[-] (1.3,4)--(4,1.3);
\draw[dash pattern= { on 2pt off 1pt}] (0,4)--(4,0);
\draw[dash pattern= { on 2pt off 1pt}] (1.6,4)--(4,1.6);
\node [below left] at (0,0) {\tiny$(0,0)$};
\filldraw[fill=black] (3.3,2)  circle[radius=0.5mm];
\filldraw[fill=black] (3.3,0.35)  circle[radius=0.3mm];
\filldraw[fill=black] (3.3,2.3)  circle[radius=0.3mm];
\filldraw[fill=black] (4,1.6)  circle[radius=0.3mm];
\filldraw[fill=black] (1.6,4)  circle[radius=0.3mm];
\draw [dotted] (3.3,2.3)--(3.3,0.35);
\draw [dotted] (1.6,4)--(3.3,0.35);
\draw [dotted] (4,1.6)--(3.3,0.35);
\draw[-] (3.3,2)--(4.5,2.5);
\node [right] at (4.5,2.5) {\tiny$(\frac{1}{p_1},\frac{1}{p_2})$};
\draw[-] (3.3,2.3)--(4.5,3);
\node [right] at (4.5,3) {\tiny$(\frac{1}{p_1},\frac{1}{r})$};
\draw[-] (3.3,0.35)--(4.5,0.6);
\node [right] at (4.5,0.6) {\tiny$C=(\frac{1}{p_1},\frac{p_1-1}{2p_1})$};
\draw[-] (3,4.5)--(2.3,3.3);
\node [right] at (3,4.5) {\tiny$x_1+x_2=\frac{1}{\wt{p}} $};
\draw[-] (2.4,5)--(2,3.3);
\node [right] at (2.4,5) {\tiny$x_1+x_2=\frac{1}{p} $};
\draw[-] (1.4,4.4)--(1.6,4);
\node [above] at (1.4,4.4) {\tiny$D=(\frac{1}{p_0},1) $};
\draw[-] (4.7,2)--(4,1.6);
\node [right] at (4.7,2) {\tiny$E=(1,\frac{1}{p_0}) $};
\node [below] at (4,0) {\tiny$(1,0)$};
\node [left] at (0,4) {\tiny$(0,1)$};
\node [below] at (5,0) {\tiny$x_1$};
\node [below] at (1.5,1.5) {\tiny Region $\mathrm{I}$};
\node [left] at (0,5) {\tiny${x_2}$};
\end{tikzpicture}
\caption{Interpolation between estimates at $C$, $D$, and $E$}\label{region2}
\end{figure}
Taking the proposition temporarily for granted, let us prove (\ref{mainidea}).
We fix $1<p_1\le p_2<\infty$ and $1/2<p<1$ satisfying $1/p=1/p_1+1/p_2$, and suppose $q>\max{(\frac{4}{3},\frac{p}{2p-1})}$.
Then we note that $1<p_1<2$ and there exists $1/2<\wt{p}<p$ such that $q>\max{(\frac{4}{3},\frac{\wt{p}}{2\wt{p}-1})}$. Choose $1<r<p_2$ so that $1/\wt{p}=1/p_1+1/r$.
Since $\frac{p_1+1}{2p_1}<1<\frac{1}{p}<\frac{1}{\wt{p}}$ we can select $0<\theta<1$ for which
\begin{equation}\label{inter1}
\frac{1}{p}=\frac{p_1+1}{2p_1}(1-\theta)+\frac{1}{\wt{p}}\;\theta
\end{equation}
and subsequently,
\begin{equation}\label{inter2}
\frac{1}{p_2}=\frac{p_1-1}{2p_1}(1-\theta)+\frac{1}{r}\theta. 
\end{equation}
Now, from the estimate (\ref{mainidea}) in Region $\mathrm{I}$, it follows that
\begin{equation}\label{-interpol}
\big\Vert \LL_{\mu} \big\Vert_{L^{p_1}\times L^{\frac{2p_1}{p_1-1}}\to L^{\frac{2p_1}{p_1+1}}}\lesssim 2^{-\epsilon_0\mu}\Vert \Omega\Vert_{L^q(\mathbb{S}^{2n-1})}\q \text{ at } ~C:=(1/p_1,(p_1-1)/(2p_1))
\end{equation}
for some $\epsilon_0>0$, as the point $(\frac{1}{p_1},\frac{p_1-1}{2p_1})$ belongs to Region $\mathrm{I}$.

On the other hand, we choose $1<p_0<\infty$  such that 
$$\frac{1}{p_0}+1=\frac{1}{\wt{p}}~(=\frac{1}{p_1}+\frac{1}{r})$$
and then there exits $0<\wt{\theta}<1$ such that
$$\frac{1-\wt{\theta}}{p_0}+\wt{\theta}=\frac{1}{p_1}\q \text{ and } {1-\wt{\theta}}+\frac{\wt{\theta}}{p_0}=\frac{1}{r}.$$
Then we observe that
\begin{equation}\label{interpolparameter}
(1-\theta)\Big(\frac{1}{p_1},\frac{p_1-1}{2p_1}\Big)+\theta (1-\wt\theta) \Big( \frac{1}{p_0},1\Big)+\theta\wt{\theta}\Big( 1,\frac{1}{p_0}\Big)=\Big( \frac{1}{p_1},\frac{1}{p_2}\Big)
\end{equation}
where $1-\theta, \theta (1-\wt{\theta})$, and $\theta \wt{\theta}$ are numbers between $0$ and $1$ which play a role of $\theta_0$, $\theta_1$, and $\theta_2$ in Lemma \ref{interpollemma}, respectively, as the sum of them is equal to $1$,
Indeed, since $$q>\max{\Big(\frac{4}{3},\frac{p_0}{p_0-1}\Big)} ~\Big(=\max{\Big(\frac{4}{3},\frac{\wt{p}}{2\wt{p}-1}\Big)} \Big),$$
it follows from Proposition \ref{pip} that for arbitrary $\delta>0$,
\begin{align}
\big\Vert \LL_{\mu} \big\Vert_{L^{p_0}\times L^{1}\to L^{\wt{p},\infty}}&\lesssim_{\delta} 2^{\delta\mu}\Vert \Omega\Vert_{L^q(\mathbb{S}^{2n-1})}\q \text{ at } ~D:=(1/p_0,1)\label{1p01},\\
\big\Vert \LL_{\mu} \big\Vert_{L^{1}\times L^{p_0}\to L^{\wt{p},\infty}}&\lesssim_{\delta} 2^{\delta\mu}\Vert \Omega\Vert_{L^q(\mathbb{S}^{2n-1})}\q \text{ at } ~E:=(1,1/p_0),\label{11p0}
\end{align}
and choosing $\delta>0$ sufficiently small and applying Lemma \ref{interpollemma} to \eqref{-interpol}, \eqref{1p01}, and \eqref{11p0}, together with \eqref{interpolparameter},
we finally obtain
\begin{equation*}
\big\Vert \LL_{\mu}\big\Vert_{L^{p_1}\times L^{p_2}\to L^{p}}\lesssim_{\epsilon_0,\theta} 2^{-\delta_0\mu}\Vert \Omega\Vert_{L^q(\mathbb{S}^{2n-1})}
\end{equation*}  for some $\delta_0>0$. This is always possible as we can choose $\delta>0$ small enough in \eqref{1p01} and \eqref{11p0} while $\epsilon_0$ is a fixed number in \eqref{-interpol}.
See Figure \ref{region2} for the interpolation.

 This ends the proof of (\ref{mainidea}) as the case $p_2<p_1$ follows via symmetry.

\hfill

Now it remains to prove Proposition \ref{pip}.
 \begin{proof}[Proof of Proposition \ref{pip}]
We are only concerned with the first inequality appealing to symmetry for the other case.
Moreover, without loss of generality, we may assume $\Vert f_1\Vert_{L^1(\bbrn)}=\Vert f_2\Vert_{L^{p_0}(\bbrn)}=\Vert \Omega\Vert_{L^{q}(\mathbb{S}^{2n-1})}=1$ and then it is enough to prove
\begin{equation}\label{weakmainest}
\Big|\Big\{    x\in \bbrn:    \big|\LL_{\mu}(f_1,f_2)(x)\big|>t    \Big\}\Big|\lesssim_{\delta} 2^{\delta \mu\frac{p_0}{p_0+1}}{t^{-\frac{p_0}{p_0+1}}}.
\end{equation}
 As in the proof of Proposition \ref{1i1},  by applying a technique of the Calder\'on-Zygmund decomposition, we write
 $$f_1=g_1+\sum_{Q\in\mathcal{A}}{b_{1,Q}}$$
 where $\mathcal{A}$ is a subset of disjoint dyadic cubes, $\big| \bigcup_{Q\in \mathcal{A}}Q\big|\lesssim {t^{-\frac{p_0}{p_0+1}}}$, $\supp(b_{1,Q})\subset Q$, $\int{b_{1,Q}(y)}dy=0$, $\Vert b_{1,Q}\Vert_{L^1(\bbrn)}\lesssim t^{\frac{p_0}{p_0+1}}|Q|$, and $\Vert g_1\Vert_{L^r}\lesssim t^{(1-\frac{1}{r})\frac{p_0}{p_0+1}}$ for all $1\le r\le \infty$.

 First of all, from Chebyshev's inequality and  the estimate (\ref{mainidea}) in Region $\mathrm{I}$, it follows  that
\begin{align*}
\Big| \Big\{x\in\bbrn: \big| \LL_{\mu}\big(g_1,f_2\big)(x)\big|>\frac{t}{2} \Big\}\Big|&\lesssim t^{-\frac{2p_0}{p_0+1}}\big\Vert \LL_{\mu}(g_1,f_2)\big\Vert_{L^{\frac{2p_0}{p_0+1}}(\bbrn)}^{\frac{2p_0}{p_0+1}}\\
&\lesssim t^{-\frac{2p_0}{p_0+1}}\Vert g_1\Vert_{L^{2p_0'}(\bbrn)}^{\frac{2p_0}{p_0+1}}\lesssim {t^{-\frac{p_0}{p_0+1}}}
\end{align*}
where the penultimate inequality follows from the $L^{2p_0'}\times L^{p_0}\to L^{\frac{2p_0}{p_0+1}}$ boundedness of $\LL_{\mu}$ as $(\frac{1}{2p_0'},\frac{1}{p_0})$ is inside Region $\mathrm{I}$. Here, $p_0'$ is the conjugate index of $p_0$.

 Since it is clear that $\big| \bigcup_{Q\mathcal{A}}Q^*\big|\lesssim {t^{-\frac{p_0}{p_0+1}}}$, the proof of (\ref{weakmainest}) can be reduced to the inequality
 \begin{equation}\label{toshowprop}
 \Big| \Big\{x\in \Big( \bigcup_{Q\in\mathcal{A}}Q^*\Big)^c: \Big| \LL_{\mu}\Big(\sum_{Q\in\mathcal{A}}b_{1,Q},f_2\Big)(x)\Big|>\frac{t}{2} \Big\}\Big|\lesssim_{\delta}2^{\delta \mu\frac{p_0}{p_0+1}}{t^{-\frac{p_0}{p_0+1}}}.
 \end{equation}
 The left-hand side of (\ref{toshowprop}) is, via Chebyshev's inequality, less than
 \begin{align*}
 &{t^{-\frac{p_0}{p_0+1}}}\int_{(\bigcup_{Q\in\mathcal{A}}Q^*)^c}\Big( \sum_{Q\in\mathcal{A}}\sum_{\gamma\in\bbz} \big|T_{K_{\mu}^{\ga}}\big(b_{1,Q},f_2 \big)(x) \big|\Big)^{\frac{p_0}{p_0+1}}dx\\
 &\le {t^{-\frac{p_0}{p_0+1}}}\int_{(\bigcup_{Q\in\mathcal{A}}Q^*)^c}\Big( \sum_{Q\in\mathcal{A}}\sum_{\gamma: 2^{\ga}\ell(Q)\ge 1} \big|T_{K_{\mu}^{\ga}}\big(b_{1,Q},f_2 \big)(x) \big|\Big)^{\frac{p_0}{p_0+1}}dx\\
 &\qq \qq \qq+{t^{-\frac{p_0}{p_0+1}}}\int_{\bbrn}\Big( \sum_{Q\in\mathcal{A}}\sum_{\gamma: 2^{\ga}\ell(Q)<1} \big|T_{K_{\mu}^{\ga}}\big(b_{1,Q},f_2 \big)(x) \big|\Big)^{\frac{p_0}{p_0+1}}dx\\
 &=:\UU_1+\UU_2.
 \end{align*}

 To estimate $\UU_1$, we see that
 \begin{align*}
 &\big|T_{K_{\mu}^{\ga}}\big(b_{1,Q},f_2 \big)(x) \big| \\
 &\lesssim\int_{(\bbrn)^2}\int_{|\zzz|\sim 2^{-\ga}}2^{2\ga n}\big| \Omega(\zzz\; ')\big| \big|\Phi_{\mu+\ga}(x-y_1-z_1,x-y_2-z_2) \big| |b_{1,Q}(y_1)| |f_2(y_2)|\; d\zzz \; d\yyy\\
 &\lesssim_L\int_{|\zzz|\sim 2^{-\ga}} 2^{2\ga n}\big| \Omega(\zzz\; ')\big| \Big(\int_{y_1\in Q}\frac{2^{(\mu+\ga)n}}{(1+2^{\mu+\ga}|x-y_1-z_1|)^L}|b_{1,Q}(y_1)|   dy_1\Big)\\
 & \qq\qq\times \Big(\int_{\bbrn}\frac{2^{(\mu+\ga)n}}{(1+2^{\mu+\ga}|x-y_2-z_2|)^L}|f_2(y_2)|   dy_2\Big) \;d\zzz
 \end{align*} for all $L>n$.
 Clearly, we have
 \begin{equation}\label{maximalbound}
 \int_{\bbrn}\frac{2^{(\mu+\ga)n}}{(1+2^{\mu+\ga}|x-y_2-z_2|)^L}|f_2(y_2)|   dy_2\lesssim \mathcal{M}f_2(x-z_2)
 \end{equation}
 and for $2^{\ga}\ell(Q)\ge 1$ and $|z_1|\le 2^{-\ga+1}$,
 $$\int_{y_1\in Q}\frac{2^{(\mu+\ga)n}}{(1+2^{\mu+\ga}|x-y_1-z_1|)^L}|b_{1,Q}(y_1)|   dy_1\lesssim \frac{2^{(\mu+\ga)n}}{(1+2^{\mu+\ga}|x-c_Q|)^L}\Vert b_{1,Q}\Vert_{L^1(\bbrn)}$$ because $|x-y_1-z_1|\gtrsim |x-c_{Q}|$.
 Therefore, we have
 \begin{align*}
 &\big|T_{K_{\mu}^{\ga}}\big(b_{1,Q},f_2 \big)(x) \big|\\
 &\lesssim \frac{2^{(\mu+\ga)n}}{(1+2^{\mu+\ga}|x-c_Q|)^L}\Vert b_{1,Q}\Vert_{L^1(\bbrn)} \int_{|\zzz|\sim 2^{-\ga}} 2^{2\ga n}\big| \Omega(\zzz \; ')\big| \mathcal{M}f_2(x-z_2) d\zzz.
 \end{align*}
Now H\"older's inequality yields
 \begin{align}\label{doublemaximal}
 &\int_{|\zzz|\sim 2^{-\ga}} 2^{2\ga n}\big| \Omega(\zzz \; ')\big| \mathcal{M}f_2(x-z_2) d\zzz \nonumber\\
 &\le \Big(\int_{|\zzz|\sim 2^{-\ga}}{2^{2\ga n}\big| \Omega(\zzz\; ')\big|^q}d\zzz \Big)^{1/q}\Big( \int_{|\zzz|\sim 2^{-\ga}}{2^{2\ga n}\big|\mathcal{M}f_2(x-z_2)\big|^{q'}}d\zzz \;'\Big)^{1/q'} \nonumber\\
 &\le \Vert \Omega\Vert_{L^q(\mathbb{S}^{2n-1})}\Big(2^{\ga n}\int_{|z_2|\lesssim 2^{-\ga}}\big| \mathcal{M}f_2(x-z_2)\big|^{q'}dz_2 \Big)^{1/q'}\lesssim \mathcal{M}_{q'}\mathcal{M}f_2(x)
 \end{align}
 and thus
 \begin{align*}
 \big|T_{K_{\mu}^{\ga}}\big(b_{1,Q},f_2 \big)(x) \big|\lesssim  \frac{2^{(\mu+\ga)n}}{(1+2^{\mu+\ga}|x-c_Q|)^L}\Vert b_{1,Q}\Vert_{L^1(\bbrn)}\mathcal{M}_{q'}\mathcal{M}f_2(x).
 \end{align*}
 This, together with H\"older's inequality, deduces that $ \UU_1$ is dominated by a constant times
 \begin{align*}
&{t^{-\frac{p_0}{p_0+1}}}\int_{(\bigcup_{Q\in\mathcal{A}}Q^*)^c} \Big(  \mathcal{M}_{q'}\mathcal{M}f_2(x)\sum_{Q\in\mathcal{A}}\sum_{\ga: 2^{\ga}\ell(Q)\ge 1}  \frac{2^{(\mu+\ga)n}}{(1+2^{\mu+\ga}|x-c_Q|)^L}\Vert b_{1,Q}\Vert_{L^1(\bbrn)}     \Big)^{\frac{p_0}{p_0+1}}dx\\
 &\le {t^{-\frac{p_0}{p_0+1}}} \big\Vert \mathcal{M}_{q'}\mathcal{M}f_2\big\Vert_{L^{p_0}(\bbrn)}^{ \frac{p_0}{p_0+1}}\Big(\sum_{Q\in\mathcal{A}}\sum_{\ga: 2^{\ga}\ell(Q)\ge 1}\Big\Vert \frac{2^{(\mu+\ga)n}}{(1+2^{\mu+\ga}|\cdot-c_Q|)^L} \Big\Vert_{L^1((Q^*)^c)}\Vert b_{1,Q}\Vert_{L^1(\bbrn)}\Big)^{\frac{p_0}{p_0+1}}.
 \end{align*}
 Since $1< q'<p_0$, which is equivalent to $q>\frac{p_0}{p_0-1}$,
 the $L^{p_0}$ norm is controlled by $\Vert f_2\Vert_{L^{p_0}(\bbrn)}=1$, using the $L^{p_0}$ boundedness of both $\mathcal{M}_{q'}$ and $\mathcal{M}$ in (\ref{HLmaximal}).
 Moreover, using the fact that for $\mu\ge \mu_0$,
 $$\Big\Vert \frac{2^{(\mu+\ga)n}}{(1+2^{\mu+\ga}|\cdot-c_Q|)^L} \Big\Vert_{L^1((Q^*)^c)}\lesssim 2^{-\mu(L-n)}\big(2^{\ga}\ell(Q) \big)^{-(L-n)}\le \big(2^{\ga}\ell(Q) \big)^{-(L-n)},$$ we have
 \begin{equation*}
 \sum_{Q\in\mathcal{A}}\sum_{\ga: 2^{\ga}\ell(Q)\ge 1}\Big\Vert \frac{2^{(\mu+\ga)n}}{(1+2^{\mu+\ga}|\cdot-c_Q|)^L} \Big\Vert_{L^1((Q^*)^c)}\Vert b_{1,Q}\Vert_{L^1(\bbrn)}\lesssim  \sum_{Q\in\mathcal{A}}\Vert b_{1,Q}\Vert_{L^1(\bbrn)}\lesssim 1.
 \end{equation*}
 This concludes 
 $$\UU_1\lesssim t^{-\frac{p_0}{p_0+1}}.$$
 
Next, we consider the other term $\UU_2$.
 By using the vanishing moment condition of $b_{1,Q}$, we write
 \begin{align}\label{tkest}
 &\big| T_{K_{\mu}^{\ga}}\big(b_{1,Q},f_2\big)(x)\big|\nonumber\\
 &\lesssim \int_{|\zzz|\sim 2^{-\ga}} 2^{2\ga n}\big| \Omega(\zzz\; ')\big| \Big(\int_{(\bbrn)^2}\big| \Phi_{\mu+\ga}(x-y_1-z_1,x-y_2-z_2)\\
 & \qq\qq\qq -\Phi_{\mu+\ga}(x-c_Q-z_1,x-y_2-z_2)\big| |b_{1,Q}(y_1)| |f_2(y_2)|d\yyy \Big) \; d\zzz. \nonumber
 \end{align}
 We observe that
 \begin{align*}
&\big| \Phi_{\mu+\ga}(x-y_1-z_1,x-y_2-z_2)-\Phi_{\mu+\ga}(x-c_Q-z_1,x-y_2-z_2)\big| \\
&\lesssim 2^{(\mu+\ga)}\ell(Q)\int_0^1\frac{2^{2(\mu+\ga)n}}{(1+2^{\mu+\ga}|x-ty_1-(1-t)c_Q-z_1|+2^{\mu+\ga}|x-y_2-z_2|)^{2L}}dt\\
&\le 2^{(\mu+\ga)}\ell(Q)V^L_{\mu+\ga}(x-z_1,y_1,c_Q) \frac{2^{(\mu+\ga)n}}{(1+2^{\mu+\ga}|x-y_2-z_2|)^L}
 \end{align*}
 where
 $$V^L_{\mu+\ga}(x,y_1,c_Q):=\int_0^1\frac{2^{(\mu+\ga)n}}{(1+2^{\mu+\ga}|x-ty_1-(1-t)c_Q|)^L}dt.$$
Furthermore,
 \begin{align*}
 &\big| \Phi_{\mu+\ga}(x-y_1-z_1,x-y_2-z_2)-\Phi_{\mu+\ga}(x-c_Q-z_1,x-y_2-z_2)\big|\\
 &\lesssim_L W_{\mu+\ga}^L(x-z_1,y_1,c_Q)\frac{2^{(\mu+\ga)n}}{(1+2^{\mu+\ga}|x-y_2-z_2|)^L}
 \end{align*}
 where $$W_{\mu+\ga}^L(x,y_1,c_Q):=\frac{2^{(\mu+\ga)n}}{(1+2^{\mu+\ga}|x-y_1|)^L}+\frac{2^{(\mu+\ga)n}}{(1+2^{\mu+\ga}|x-c_Q|)^L}.$$

By averaging these two estimates and letting
$$U_{\mu+\ga}^{L,\delta}(x,y_1,c_Q):=\big( V_{\mu+\ga}^L(x,y_1,c_Q)\big)^{\delta}\big(W_{\mu+\ga}^{L}(x,y_1,c_Q) \big)^{1-\delta},$$
 we obtain 
\begin{align}\label{differenceest}
&\big| \Phi_{\mu+\ga}(x-y_1-z_1,x-y_2-z_2)-\Phi_{\mu+\ga}(x-c_Q-z_1,x-y_2-z_2)\big| \nonumber\\
&\lesssim_{L,\delta}2^{\delta \mu}\big(2^{\ga}\ell(Q) \big)^{\delta}    U_{\mu+\ga}^{L,\delta}(x-z_1,y_1,c_Q)         \frac{2^{(\mu+\ga)n}}{(1+2^{\mu+\ga}|x-y_2-z_2|)^L}.
\end{align}
 Here, we note that
 \begin{equation}\label{umugaest}
 \big\Vert U_{\mu+\ga}^{L,\delta}(\cdot, y_1,c_Q)\big\Vert_{L^1(\bbrn)}\le \big\Vert V_{\mu+\ga}^{L}(\cdot,y_1,c_Q)\big\Vert_{L^1(\bbrn)}^{\delta} \big\Vert W_{\mu+\ga}^{L}(\cdot,y_1,c_Q)\big\Vert_{L^1(\bbrn)}^{1-\delta}\lesssim 1.
 \end{equation}
 
 By plugging (\ref{differenceest}) into (\ref{tkest}), we obtain
 \begin{align*}
 &\big| T_{K_{\mu}^{\ga}}\big(b_{1,Q},f_2\big)(x)\big|\\
 & \lesssim 2^{\delta \mu}\big( 2^{\ga}\ell(Q)\big)^{\delta}\int_{|\zzz|\sim 2^{-\ga}}2^{2\ga n}\big| \Omega(\zzz\; )\big|\Big(\int_{\bbrn}{U_{\mu+\ga}^{L,\delta}(x-z_1,y_1,c_Q) |b_{1,Q}(y_1)|}dy_1 \Big)\\
 &\qq\qq \qq\qq \qq \times \Big(\int_{\bbrn}{\frac{2^{(\mu+\ga)n}}{(1+2^{\mu+\ga}|x-y_2-z_2|)^L} |f_2(y_2)|}dy_2 \Big)\; d\zzz\\
 &\lesssim 2^{\delta \mu}\big( 2^{\ga}\ell(Q)\big)^{\delta}
 \int_{|z_1|\lesssim 2^{-\ga}}2^{\ga n}\int_{\bbrn}{U_{\mu+\ga}^{L,\delta}(x-z_1,y_1,c_Q) |b_{1,Q}(y_1)|}dy_1\\
 &\qq\qq \qq\qq \qq \times \int_{|z_2|\lesssim 2^{-\ga}}2^{\ga n}\big| \Omega(\zzz\; )\big|\mathcal{M}f_2(x-z_2)\; dz_2 dz_1\\
 &\lesssim 2^{\delta \mu}\big( 2^{\ga}\ell(Q)\big)^{\delta}
 \int_{|z_1|\lesssim 2^{-\ga}}2^{\ga n}\int_{\bbrn}{U_{\mu+\ga}^{L,\delta}(x-z_1,y_1,c_Q) |b_{1,Q}(y_1)|}dy_1\\
 &\qq\qq \qq\qq \qq \times \Big(\int_{|z_2|\lesssim 2^{-\ga}}2^{\ga n}\big| \Omega(\zzz\; )\big|^q\; dz_2\Big)^{1/q} \;\mathcal{M}_{q'}\mathcal{M}f_2(x)dz_1\\
 &\lesssim 2^{\delta \mu}\big( 2^{\ga}\ell(Q)\big)^{\delta} \;\mathcal{M}_{q'}\mathcal{M}f_2(x)
\\ 
 &\qq   \int_{\bbrn} |b_{1,Q}(y_1)| \int_{|z_1|\lesssim 2^{-\ga}}2^{\ga n}{U_{\mu+\ga}^{L,\delta}(x-z_1,y_1,c_Q) }
\Big(\int_{|z_2|\lesssim 2^{-\ga}}2^{\ga n}\big| \Omega(\zzz\; )\big|^q\; dz_2\Big)^{1/q} dz_1dy_1
 \end{align*} 
 where (\ref{maximalbound}) is applied. 
It follows from last control and \eqref{umugaest} that 
$\UU_2^{\tf{p_0+1}{p_0}}$ is bounded by 
\begin{align*}
&t^{-1}2^{\delta \mu} \Big\Vert\;\mathcal{M}_{q'}\mathcal{M}f_2\times \sum_{Q\in\mathcal{A}}\sum_{\gamma: 2^{\ga}\ell(Q)<1} \big( 2^{\ga}\ell(Q)\big)^{\delta}\\
&
\int_{\bbrn} |b_{1,Q}(y_1)|
\int_{|z_1|\lesssim 2^{-\ga}}2^{\ga n}{U_{\mu+\ga}^{L,\delta}(\cdot-z_1,y_1,c_Q) }
\Big(\int_{|z_2|\lesssim 2^{-\ga}}2^{\ga n}\big| \Omega(\zzz\; )\big|^q\; dz_2\Big)^{1/q} dz_1dy_1
\Big\Vert_{L^{\tf{p_0}{p_0+1}}(\rn)}\\
\lesssim&
t^{-1}2^{\delta \mu}\Vert f_2\Vert_{L^{p_0}}
\Big\Vert  \sum_{Q\in\mathcal{A}}\sum_{\gamma: 2^{\ga}\ell(Q)<1} \big( 2^{\ga}\ell(Q)\big)^{\delta}\\
&
\int_{\bbrn} |b_{1,Q}(y_1)|
\int_{|z_1|\lesssim 2^{-\ga}}2^{\ga n}{U_{\mu+\ga}^{L,\delta}(\cdot-z_1,y_1,c_Q) }
\Big(\int_{|z_2|\lesssim 2^{-\ga}}2^{\ga n}\big| \Omega(\zzz\; )\big|^q\; dz_2\Big)^{1/q} dz_1dy_1
\Big\Vert_{L^{1}(\rn)}\\
\lesssim&t^{-1}2^{\delta \mu}\Vert f_2\Vert_{L^{p_0}}
\sum_{Q\in\mathcal{A}}\sum_{\ga:2^{\ga}\ell(Q)<1}\big( 2^{\ga}\ell(Q)\big)^{\delta}
\\&
\int_{\bbrn} |b_{1,Q}(y_1)|
\int_{|z_1|\lesssim 2^{-\ga}}2^{\ga n}\int_{\rn}|U_{\mu+\ga}^{L,\delta}(x-z_1,y_1,c_Q) |dx
\Big(\int_{|z_2|\lesssim 2^{-\ga}}2^{\ga n}\big| \Omega(\zzz\; )\big|^q\; dz_2\Big)^{1/q} dz_1dy_1\\
\lesssim&t^{-1}2^{\delta \mu}\Vert f_2\Vert_{L^{p_0}}
 \sum_{Q\in\mathcal{A}}\sum_{\ga:2^{\ga}\ell(Q)<1}\big( 2^{\ga}\ell(Q)\big)^{\delta}
\Vert b_{1,Q}\Vert_{L^1(\rn)}
\Big(\int_{|z_1|\lesssim 2^{-\ga}}
\int_{|z_2|\lesssim 2^{-\ga}}2^{2\ga n}\big| \Omega(\zzz\; )\big|^q\; dz_2dz_1\Big)^{1/q} \\
\lesssim & t^{-1}2^{\delta \mu}
\sum_{Q\in\mathcal{A}}\Vert b_{1,Q}\Vert_{L^1(\bbrn)}\\
\lesssim & t^{-1}2^{\delta \mu}
\end{align*}
 where we apply the maximal inequality for $\mathcal{M}_{q'}$ and $\mathcal{M}$.
 
 This concludes the proof of (\ref{toshowprop}).
  \end{proof}

 \section*{Appendix : Proof of Lemma \ref{05031}}
We define
$$
\mathcal D_\la^1:=\big\{\om \in \mathcal{D}_{\la}: \supp( \om )\cap \{\xxxi\in (\bbrn)^2: 2^{-j}|\xi_1|\le|\xi_2|\le 2^j|\xi_1|\}\neq\eset \big\},
$$
and $\mathcal D^2_\la:=\mathcal D_\la\setminus \mathcal D_\la^1$. 
Correspondingly, we define 
$$m_0^i:=\sum_\la m^{i,\la}=\sum_{\la}\sum_{\om\in\mathcal D_\la^i}a_\om \om,$$ 
$m_k^i:=m_0^i(2^{-k}\;\vec{\cdot}\;)$, and $m^i:=\sum_{k\in\mathbb Z} m^i_k$ for $i=1,2$.
We can decompose $m_0=m_0^1+m_{0}^2$, and  $T_{m_0}= T_{m_0^1}+T_{m_0^2}$ where we recall that $T_m$ is the bilinear multiplier operators associated to $m$.

Due to the properties (i) and (ii), we obtain from a bilinear Plancherel-type estimates \cite[Proposition 2.2]{Paper1} (see also \cite[Section 3]{Gr_He_Sl2020}) that 
$$
\big\| T_{m^{1,\la}}(f_1,f_2)\big\|_{L^1(\bbrn)}\lesssim B^{1-\tf r4}2^{-\la (N(1-\tf r4)-C(n,r)\tf r4)}\|f_1\|_{L^2(\bbrn)}\|f_2\|_{L^2(\bbrn)}.
$$
Summing over $\la$, we have
$$
\big\| T_{m_0^1}(f_1,f_2)\big\|_{L^1(\bbrn)}\lesssim B^{1-\tf r4}\|f_1\|_{L^2(\bbrn)}\|f_2\|_{L^2(\bbrn)}
$$
as $N\gg \tf r{4-r}C(n,r)$.
Since $m^1_k=m^1_0(2^{-k}\;\vec{\cdot}\;)$, and $m_0^1$ is supported in $E^j\times E^j$ with $E^j=\{\xi\in\rn:\ 1\le |\xi|\le 2^{j+2}\}$, a standard dilation argument shows that 
$$
\big\| T_{m_k^1}(f_1,f_2)\big\|_{L^1(\bbrn)}\lesssim B^{1-\tf r4}\|f_{1,k}\|_{L^2(\bbrn)}\|f_{2,k}\|_{L^2(\bbrn)},
$$
where $\wh{f_{i,k}}:=\wh f_i\chi_{E^j_k}$ and $E^j_k:=\{\xi\in\rn:\ 2^k\le |\xi|\le 2^{j+k+2}\}$. Summing over $k$ and using the (almost) orthogonality for $f_{i,k}$ in $L^2$, we obtain 
\begin{align}\label{e05041}
\big\| T_{m^1}(f_1,f_2)\big\|_{L^1(\bbrn)}&\lesssim B^{1-\tf r4}\sum_{k\in\mathbb Z}\| f_{1,k} \|_{L^2(\bbrn)} \|f_{2,k}\|_{L^2(\bbrn)}\nonumber\\
&\lesssim B^{1-\tf r4}\Big( \sum_{k\in\mathbb Z}\| f_{1,k} \|_{L^2(\bbrn)}^2\Big)^{1/2} \Big( \sum_{k\in\mathbb{Z}}\|f_{2,k}\|_{L^2(\bbrn)}^2\Big)^{1/2}\nonumber\\
&\lesssim jB^{1-\tf r4}\|f_1\|_{L^2(\bbrn)}\|f_2\|_{L^2(\bbrn)}.
\end{align}

\medskip

The operator $T_{m^2}$ has already been handled by \cite[Section 6.1]{Paper1} with $m=2$, which goes back to \cite[Section 5]{Gr_He_Ho2018}. 
We provide only the outline of the proof for this case.

We may further assume that $|\xi_2|\le 100$ in the support of $m_0^2$ by symmetry. As a consequence, the number of $\nu_2$ in $\mathcal D_\la^2$ is at most $C2^{\la n}$. 
One easily verifies that the Fourier transform of $T_{m_{k}^2}(f_1,f_2)$ is supported in $\{\xi\in\rn:\ 2^{j+k-10}\le |\xi|\le 2^{j+k+10}\}$, so by the square function characterization of Hardy spaces, we see that 
\begin{equation}\label{e05042}
\big\| T_{m^2}(f_1,f_2)\big\|_{L^1(\bbrn)}= \Big\| \sum_k T_{m_k^2}(f_1,f_2)\Big\|_{L^1(\bbrn)}\lesssim \Big\| \Big( \sum_k \big| T_{m_k^2}(f_1,f_2)\big|^2\Big)^{1/2} \Big\|_{L^{1}(\bbrn)}.
\end{equation}
See, for instance, \cite[Section 6.1]{Paper1} for more details.

Let us define 
$$m_{k}^{2,\la}=m^{2,\la}(2^{-k}\; \vec{\cdot}\;)=\sum_{\om\in\mathcal D^2_\la}a_\om \om(2^{-k}\;\vec{\cdot}\;)$$ 
and $(L_{\nu,k}f)^{\wedge}:=\om_{\nu}(2^{-k}\cdot) \wh f$ for $\nu\in\bbzn$.
 We observe that 
$L_{\nu,k}f(x)\lesssim 2^{\la n/2} \mathcal{M}f(x)$ for all $\nu\in \bbzn$
where we recall that $\mathcal{M}$ is the Hardy-Littlewood maximal operator.
 Then the right-hand side of \eqref{e05042} is bounded by 
\begin{align*}
&\Big\| \Big( \sum_k\Big| \sum_\la\sum_{\nnu\in\mathcal D^2_\la}a_{\om_{\nnu}} L_{\nu_1,k}f_1 L_{\nu_2,k}f_2\Big|^2 \Big)^{1/2}\Big\|_{L^1(\bbrn)}\\
\le&\sum_\la\sum_{\nu_2}\Big\|\Big(\sum_k \Big|\sum_{\nu_1}a_{\om_{\nnu}} L_{\nu_1,k}f_1 L_{\nu_2,k}f_2\Big|^2 \Big)^{1/2}\Big\|_{L^1(\bbrn)}\\
\le &\sum_\la\sum_{\nu_2}\Big\|\Big( \sum_k\Big|\sum_{\nu_1}a_{\om_{\nnu}} L_{\nu_1,k}f_1\Big|^2\Big)^{1/2} 2^{\la n/2}\mathcal{M}f_2\Big\|_{L^1(\bbrn)}\\
\le & \sum_\la2^{\la n/2}\sum_{\nu_2}\Big\|\Big( \sum_k\Big| \sum_{\nu_1}a_{\om_{\nnu}} L_{\nu_1,k}f_1\Big|^2\Big)^{1/2} \Big\|_{L^2(\bbrn)}\|f_2\|_{L^2(\bbrn)}
\end{align*}
where we applied the Cauchy-Schwarz inequality and the maximal inequality (\ref{HLmaximal}) in the last estimate.
Since $(L_{\nu_1,k}f_1)^{\wedge}$ is supported in $B(0,2^{j+k+3})\setminus B(0,2^{j+k-3} )$, using Plancherel's identity, we control last expression by 
$$
\sum_\la 2^{\la n/2}\sum_{\nu_2}B 2^{-\la N}\|f_1\|_{L^2(\bbrn)}\|f_2\|_{L^2(\bbrn)}
\lesssim B \|f_1\|_{L^2(\bbrn)}\|f_2\|_{L^2(\bbrn)}
$$
since the number of $\nu_2$ in $\mathcal D_\la^2$ is at most $2^{\la n}$ and $N$ is sufficiently large.
This combined with \eqref{e05041}  gives  \eqref{e05032}.

 \section*{Statements and Declarations} 
 {\bf Conflict of interest}\\
On behalf of all authors, the corresponding author states that there is no conflict of interest.

 {\bf Data availability}\\
Data sharing not applicable to this article as no datasets were generated or analysed during the current study.



\end{document}